\documentclass[11pt,reqno,a4paper]{amsart}

\usepackage{epsfig,color,xcolor}
\usepackage[active]{srcltx}
\usepackage[colorlinks,pdfpagelabels,pdfstartview=FitH,bookmarksopen=true,bookmarksnumbered=true,linkcolor=blue,plainpages=false,hypertexnames=false,citecolor=red]{hyperref}
\usepackage{mathrsfs}
\usepackage{verbatim, amssymb, amsmath, amsfonts, amsthm}
\usepackage{latexsym}

\allowdisplaybreaks 

\newtheorem{theorem}{Theorem}[section]
\newtheorem{proposition}[theorem]{Proposition}
\newtheorem{lemma}[theorem]{Lemma}
\newtheorem{corollary}[theorem]{Corollary}
\newtheorem{remark}[theorem]{Remark}

\numberwithin{equation}{section}



\newcommand{\fr}{\frac }
\newcommand{\N}{\mathbb{N}}
\newcommand{\R}{\mathbb{R}}

\def\sideremark#1{\ifvmode\leavevmode\fi\vadjust{\vbox to0pt{\vss
			\hbox to 0pt{\hskip\hsize\hskip1em
				\vbox{\hsize1.1cm\tiny\raggedright\pretolerance10000
					\noindent #1\hfill}\hss}\vbox to5pt{\vfil}\vss}}}%

\begin{document}

\title[Morse index and uniqueness]{
		Morse index and uniqueness  \\of positive solutions of the\\ Lane-Emden problem in planar domains}
	
\author[]{F. De Marchis,  M. Grossi, I. Ianni, F. Pacella}
	
\address{Franesca De Marchis, Massimo Grossi, Filomena Pacella, Dipartimento di Matematica,  Universit\`a \emph{Sapienza}, P.le Aldo Moro 5, 00185 Roma, Italy}
\address{Isabella Ianni, Dipartimento di Matematica e Fisica, Universit\`a degli Studi della Campania \emph{Luigi Vanvitelli}, V.le Lincoln 5, 81100 Caserta, Italy}
	
\thanks{2010 \textit{Mathematics Subject classification:}  }
	
	
\thanks{Research partially supported by: PRIN $2015$KB$9$WPT grant and INDAM - GNAMPA}

\begin{abstract} 
		We compute the Morse index  of \emph{$1$-spike solutions} of the semilinear elliptic  problem
		\begin{equation}
		\label{abstr}\tag{$\mathcal P_p$}
		\left\{\begin{array}{lr}-\Delta u= u^p &  \mbox{ in }\Omega\\
		u=0 &\mbox{ on }\partial \Omega
		\\
		u>0 & \mbox{ in }\Omega
		\end{array}\right.
		\end{equation}
		where $\Omega\subset \R^2$ is a smooth bounded domain and $p>1$ is sufficiently large.
		\\
		When $\Omega$ is convex, our result, combined with the characterization in \cite{DIPpositive}, a result in \cite{Lin} and with recent uniform estimates in \cite{Sirakov}, gives the uniqueness of the solution to \eqref{abstr}, for $p$  large.
		This proves, in dimension two and for $p$ large, a conjecture by Gidas-Ni-Nirenberg \cite{GNN}.
\end{abstract}

\maketitle

\section{Introduction}
	We consider the Lane-Emden Dirichlet problem 
	\begin{equation}\label{problem}
	\left\{\begin{array}{lr}-\Delta u= u^p &  \mbox{ in }\Omega\\
	u=0 &\mbox{ on }\partial \Omega
	\\
	u>0 & \mbox{ in }\Omega
	\end{array}\right. 
	\end{equation}
	where $p>1$ and $\Omega\subset \R^2$ is a smooth bounded domain. 
	\\
	This work focuses on the following issues:  \emph{the problem of computing the Morse index and studying the nondegeneracy} of solutions of \eqref{problem} in any general domain $\Omega$  and \emph{the question of the uniqueness of the solution of \eqref{problem}   when the domain $\Omega$ is convex}. As we will see the two topics are strictly related.
	\\
	\\
	We recall that in any smooth bounded domain $\Omega\subset \mathbb R^N$, $N\geq 2$,  problem \eqref{problem} admits at least one solution for any $p>1$ (and $p<\frac{N+2}{N-2}$, when $N\geq 3$), which can be obtained by standard variational methods, for example minimizing the associated energy functional on the Nehari manifold. Such a solution is usually called  \emph{least energy solution}.
	\\
	\\
	Uniqueness or multiplicity results are known, depending on the domain $\Omega$ and on the value of $p$. 
	\\
	\\
	When $p$ is close enough to $1$ it is known that the  solution is  unique and nondegenerate, in any domain $\Omega\subset \mathbb R^N$	(see \cite{DamascelliGrossiPacella, DancerMa2003,  Lin}). 
	\\\\
	When the domain is a ball, as a consequence of the famous symmetry result by Gidas, Ni and Nirenberg \cite{GNN}  it follows that any solution of \eqref{problem}  is radial, and then  one immediately gets the  uniqueness of the solution by ODE techniques.  
	\\
	\\
	In \cite{GNN} it has been conjectured that \emph{\eqref{problem} admits a unique solution in any convex domain}, in any dimension $N\geq 2$, as long as $p$ is such that a solution exists. 
	A complete answer to this conjecture has not been given so far, while partial results are available as we describe below.
	\\
	\\
	Notice that  there are non-convex domains for which multiple solutions to \eqref{problem} exist. The typical case is the annulus or more general annular domains (see for example \cite{CaWa, GGPS, YYLi} and \cite{BCGP} where a new type of positive solution is constructed). We also quote \cite{EspositoMussoPistoiaPos} for not simply connected planar domains and \cite{Dancer88} for  dumb-bell shaped domains.
	\\
	\\
	On the other side convexity is not necessary for uniqueness, for example for $p$ close to $1$ as we have already recalled. Another case is considered in \cite{Zou} where uniqueness has been proved (in dimension $N\geq 3$) when the domain $\Omega$ is a suitable perturbation of the ball, not necessarily convex.  Moreover uniqueness and nondegeneracy of the solution to \eqref{problem} hold in any domain $\Omega\subset\mathbb R^N$  which is symmetric and convex with respect to $N$ orthogonal directions, for every  $p>1$ if $N=2$ (\cite{DamascelliGrossiPacella, Dancer88}) and for $p$ slightly subcritical if  $N\geq 3$ (\cite{GrossiADE2000}). Note that these domains do not need to be convex. 
	\\
	\\
	We also recall that in dimension $N=2$, when the domain is $\Omega=(0,1)^2$ and $p=2,3$ the uniqueness and nondegeneracy are obtained in \cite{Computer, Computer2} via a computer assisted proof, even in the case when a linear term is added, giving a precise description of the solution.
	\\
	\\
	Observe that \emph{nondegeneracy is a sufficient condition for uniqueness} in any dimension $N\geq 2$. Indeed, as pointed out in \cite{DamascelliGrossiPacella, Lin}, uniqueness easily follows by extending each possible branch of  nondegenerate solutions up to  $p$ close to $1$ and then  exploiting the uniqueness result already known for $p$ in that range.
	\\
	\\
	This approach was pursued by Lin (\cite{Lin}), who proved the uniqueness of the  least energy solution to \eqref{problem}  for any $p>1$ in any convex domain when the dimension is $N=2$. The proof consists in observing that any least energy solution has \emph{Morse index equal to $1$} (this is always true in any $\Omega\subset\mathbb R^N,$ $N\geq 2$) and then in showing that, \emph{when the domain is convex  and $N=2$, any Morse index $1$-solution to \eqref{problem} is nondegenerate}. This gives the uniqueness of the branch of Morse index $1$-solutions.
	\\
	\\
	Let us recall that the \emph{Morse index} of a solution $u_p$ to \eqref{problem} can be defined as
	\[m(u_p):= \# \{k\in\mathbb N\ :\ \lambda_{k,p}<1\},\]
	where $\lambda_{1,p}<\lambda_{2,p}\leq\lambda_{3,p}\leq \ldots$ is the sequence of eigenvalues  for the problem
\begin{equation}
	\label{linea}
	\left\{
	\begin{array}{ll}
	-\Delta v=\lambda pu_p^{p-1} v & \mbox{in }\Omega\\
	v=0 & \mbox{on }\partial\Omega
	\end{array}
	\right..
\end{equation}
	When  $\lambda =1$ the space of solutions $v$ to \eqref{linea}  is the kernel of the linearized operator at  $u_p$, hence a solution $u_p$ is degenerate iff $\lambda_{k,p}=1$ for some $k$; in this case it is useful to define also the \emph{augmented Morse index} of $u_p$:
	\[m_0(u_p):= \# \{k\in\mathbb N\ :\ \lambda_{k,p}\leq 1\},\]
	clearly $m(u_p)=m_0(u_p)$ for a nondegenerate solution $u_p$.
	\\
	\\
We point out that the Morse index gives a strong qualitative information on the solutions. It enters in detecting  symmetries, singularities, nodal sets as well as classifying solutions  (see e.g.   \cite{AftalionPacella, BahriLions, DancerIndiana2004,  DancerTAMS2005, Farina, Pacella, PacellaWeth2007, YangJFA98}). 
	\\\\
Let us observe that, since $p>1$, the first eigenvalue of \eqref{linea}, which  for any solution $u_p$ to \eqref{problem} is $\lambda_{1,p}=\frac{1}{p}$, is always less than $1$, hence $m(u_p)\geq 1$.
In general to get an exact computation of the Morse index  is not an easy task since it involves precise information on the spectrum of a linear Schr\"odinger type operator.
	\\
	\\
	In dimension $N\geq 3$ some results are available when $p$ is slightly subcritical, namely $p=p_{\varepsilon}=\frac{N+2}{N-2}-\varepsilon$, with $\varepsilon>0$  a small parameter. In this case it is well known (\cite{BahriLiRey, Han,  Rey, Schoen, Struwe}) that any bounded sequence in $H^1_0(\Omega)$ of solutions 
	to \eqref{problem}  (up to a subsequence) either converges  as $\varepsilon\rightarrow 0$
	to a positive solution of the critical problem (if any), or it blows-up at $k$ points $x_{1,\infty},x_{2,\infty},\ldots,x_{k,\infty} \in \Omega$. In this second case the  sequences of solutions which blow-up are usually called \emph{multi-bubbles solutions} since they 
	can be approximated in $H^1_0(\Omega)$ as the $\sum_{j=1}^k P\delta_{j,\varepsilon}$, where $P\delta_{j,\varepsilon}$ is the  projections onto $H^1_{0}(\Omega)$ of the  function $\delta_{j,\varepsilon}(x):= \eta_{j,\varepsilon}^{\frac{N+2}{N-2}}   \left(   1+\eta_{j,\varepsilon}^2|x-x_{j,\infty}|^2\right)^{-\frac{N+2}{N-2}}$, for certain positive parameters $\eta_{j,\varepsilon}\rightarrow +\infty$ as $\varepsilon \rightarrow 0$.
	The points $x_{j,\infty}$  are critical points of an explicit function which involves Green and Robin function of the Dirichlet Laplacian in $\Omega$ and the total energy converges to the number $k\cdot S$, where $S$ is the best Sobolev constant.
	\\
	In \cite{BahriLiRey} and \cite{Rey} Bahri-Li-Rey for $N\geq 4$  and Rey  for $N=3$ (see also \cite{ChoiKimLee}) proved that the Morse index of a \emph{$k$-bubble solution} $u_{p_{\varepsilon}}$ is determined, for $\varepsilon$ small, by the number  $\ell$ of the negative eigenvalues (or the number $\ell_0$ of the non-positive eigenvalues) of a certain symmetric matrix whose entries are a combination of Green function, Robin function, and their first and second derivatives at the blow-up points $x_{j,\infty}$, according to the following chain of inequalities:
	\[ k+\ell\leq m(u_{p_{\varepsilon}})\leq m_0(u_{p_{\varepsilon}})\leq k+\ell_0\leq k (1+N).\]
	Hence when $\ell=\ell_0$  the solution   $u_{p_{\varepsilon}}$ is nondegenerate  and
	\[m(u_{p_{\varepsilon}})=k+\ell\]
	for $\varepsilon$ sufficiently small.	
	\\
	\\
	In dimension $N=2$ there are not many results about the computation of the Morse index for the solutions to \eqref{problem}, the main reason being the lack of a complete  understanding of the  properties of the solutions to \eqref{problem}, in dependence on the exponent  $p$.
	\\
	The $2$-dimensional case is indeed different, in particular there is no Sobolev critical exponent and the study of the asymptotic behavior of solutions as $p\rightarrow +\infty$, which could be seen as the counterpart of the asymptotic analysis for $N\geq 3$, was not carried out until recently with the exception of the special case of least energy solutions  (see \cite{AdimurthiGrossi, RenWeiTAMS1994} and \cite{RenWeiPAMS1996}). \\
	In \cite{DIPJEMS,  DIPpositive, DGIPsqrte}  a fairly complete characterization of the asymptotic behavior as $p\rightarrow +\infty$ of the solutions to \eqref{problem}, even in the case of sign-changing solutions, has been obtained. 
The authors consider families of solutions $u_p$ to \eqref{problem} in a general smooth bounded domain $\Omega\subset\mathbb R^2$ which satisfy the uniform bound
	\begin{equation}
	\label{energylimit}
	\sup_{p}\ p \|\nabla u_p\|^2_{2}\leq C
	\end{equation}
	and show that in the limit as $p\rightarrow +\infty$ these  $u_p$ are necessarily  \emph{multi-spike solutions}. 
More precisely, differently from the higher dimensional case,  they do not blow-up  and there exists a finite number $k$   of distinct  points $x_{j,\infty}\in\Omega$, $j=1,\ldots, k$   and a sequence $p_n\rightarrow +\infty$ as $n\rightarrow +\infty$ such that $u_{p_n}$ concentrate at the set 
\[\mathcal S:=\{x_{1,\infty},\ldots, x_{k,\infty}\}.\]
Moreover
\[\lim_{n\rightarrow +\infty}\max_{\overline{B_{\delta}(x_{j,\infty})}} u_{p_n}= \sqrt{e} \qquad\mbox{for small $\delta>0$}\]
		and 
\[	\lim_{n\rightarrow +\infty}p_nu_{p_n}= 8\pi \sqrt{e} \sum_{j=1}^k G(\cdot,x_{j,\infty})\  \mbox{ in } 
		C^2_{loc}(\overline\Omega\setminus\mathcal S), 	\]
where   $G$ is the Green function of $-\Delta$ in $\Omega$ under Dirichlet boundary conditions. Furthermore the location of the concentration points depends on 
$G$ and its regular part $H$
		according to the following system
\begin{equation}\label{x_j relazioneTeoIntro}
		-\nabla_x H(x_{j,\infty},x_{j,\infty})+\sum_{j\neq h}\nabla_x G(x_{j,\infty},x_{h,\infty})=0.		
\end{equation}
		In \cite{DGIPsqrte} it is also proved that the total energy is quantized to integer multiple of $8\pi e$:
		\[\lim_{n\rightarrow +\infty}p_n\|\nabla u_{p_n}\|^2_2=k\cdot 8\pi e,\] 
and  for the whole family $u_p$ it holds:
		\begin{equation}
		\label{NormaLinf}
		\lim_{p\rightarrow +\infty}\|u_{p}\|_{L^{\infty}(\Omega)}=\sqrt{e}.
		\end{equation}
		We refer to  Section \ref{section:preliminaries} for all the details about the asymptotic results, in particular see Theorem \ref{teoAsymptotic}. When $u_{p_n}$ is a sequence of solutions satisfying all the properties in Theorem \ref{teoAsymptotic} with $k=1$, we will call it simply a \emph{$1$-spike sequence of solutions}.
	\\
	\\
	In this paper, by exploiting this asymptotic analysis,  we compute the Morse index of the $1$-spike solutions to \eqref{problem}  for large values $p_n$, getting so a result analogous to that 
	obtained by Bahri-Li-Rey \cite{BahriLiRey} and Rey  \cite{Rey} for the case of $1$-bubble solutions in dimension $N\geq 3$.\\ 
	Observe that in this case  the  system \eqref{x_j relazioneTeoIntro} reduces to the single equation \[\nabla_x H(x_\infty,x_\infty)=0,\] where $x_{\infty}:=x_{1,\infty}$, namely $x_\infty\in \Omega$ is a critical point of the  Robin function $x\mapsto R(x):=H(x,x)$. 
	\\	
	The function  $R$ is $C^2$ and we denote by  $m(x_\infty)$ and  $m_0(x_\infty)$ the Morse and augmented Morse index of  $x_\infty$, as a critical point of $R$, that is:
	\[\left\{
	\begin{array}{lr}
	m(x_\infty):= \# \{k\in\{1,2\}\ :\ \mu_k<0\}
	\\
	m_0(x_\infty):=\# \{k\in\{1,2\}\ :\ \mu_k\leq0\}
	\end{array}
	\right.
	\]
	where $\mu_1\leq \mu_2$ are the eigenvalues of the hessian matrix $D^2 R(x_{\infty})$ of the Robin function $R$ at the point $x_\infty$.
	\\
	\\
	Our first  result is the following: 
	\begin{theorem}
		\label{theorem:Morse1bubble}
		Let  $(u_{p_n})_n$ be a $1$-spike sequence of solutions to \eqref{problem}. Then there exists $n^{*}\in\mathbb N$ such that
		\begin{equation}\label{diesis} 1+m(x_{\infty})\leq m(u_{p_n})\leq m_0(u_{p_n})\leq 1+m_0(x_\infty)\leq  2, \quad \forall n\geq n^{*}.\end{equation}
		Moreover if $x_{\infty}$ is nondegenerate then  $u_{p_n}$ is nondegenerate for $n\geq n^*$ and
		\[m(u_{p_n})=1+m(x_{\infty}), \quad \forall n\geq n^{*}.\]	
	\end{theorem}
	Observe that by \eqref{diesis} $m(u_{p_n})$ can be either  $1$ or $2$. This is due to the estimate $m_0(x_{\infty})\leq 1$, which is a consequence of the properties of the Robin function in planar domains (see Section \ref{section:preliminaries}). In higher dimension $N$ one has instead the weaker estimate $m_0(x_{\infty})\leq N$.
	\\
	The proof of Theorem \ref{theorem:Morse1bubble} is quite long and consists in analyzing the asymptotic behavior of the eigenvalues and eigenfunctions of the linearized operator at the \emph{$1$-spike solutions} $u_{p_n}$  by exploiting all the information collected in Theorem \ref{teoAsymptotic}.  
	A similar approach can be found in \cite{GrossiPacellaMathZ} for the almost critical problem in dimension $N\geq 3$ (see also \cite{GladialiGrossi} where the Gelfand problem is studied). 
	However the analysis of the planar Lane-Emden problem for large exponent is more delicate and several additional estimates are required.\\ 
	In the case of  sign-changing radial solutions in the ball the  Morse index has been computed in  \cite{DeMarchisIanniPacellaMathAnn, DeMarchisIanniPacellaAdvances} following a different approach which combines the information on the asymptotic behavior of the solutions with a spectral decomposition. This methods cannot be used for general non-symmetric solutions. 
	\\
	The precise asymptotic  behavior, as $n\rightarrow +\infty$, of the eigenvalues $\lambda_{i,p_n}$ and eigenfunctions $v_{i,p_n}$ for $i=2,3,4$, of the linearized operator $L_{p_n}=-\Delta -p_n u_{p_n}^{p_n-1}$ at the $1$-spike sequence of solution $u_{p_n}$ is described in the following theorem
	\begin{theorem}
		\label{teoAutofunzioniAutovalori}
		Under the same assumptions of Theorem \ref{theorem:Morse1bubble} one has, as $n\to+\infty$,
		\begin{eqnarray}		
		\label{ele2}
		&&	\frac{v_{i,p_n}}{\varepsilon_n}= 2\pi\sum_{j=1}^2 a^i_j\frac{\partial G}{\partial x_j}(\cdot,x_\infty)+ o(1)\qquad \mbox{in $C^1_{loc}(\overline\Omega\setminus\{x_\infty\})$}\qquad i=2,3
		\\ 
		\label{ele1}
		&&		
		\lambda_{i,p_n} =1+24\pi \varepsilon_{n}^2\mu_{i-1} +o(\varepsilon_n^2)\qquad i=2,3
		\\
		\label{ele4}
		&&	p_n v_{4,p_n}= 4\pi b \,G(\cdot,x_\infty)+o(1)\qquad \mbox{in $C^1_{loc}(\overline\Omega\setminus\{x_\infty\})$}
		\\
		\label{ele3}	
		&&	\lambda_{4,p_n}= 1+\frac{6}{p_n}+o(\frac{1}{p_n})
		\end{eqnarray}
		where $\varepsilon_n:=[p_n\|u_{p_n}\|_{L^{\infty}(\Omega)}^{p_n-1}]^{-1/2}\rightarrow 0$ as $n\rightarrow +\infty$ by \eqref{NormaLinf},  $\mu_1\leq\mu_2$ are  the eigenvalues of the hessian  matrix $D^2 R(x_\infty)$,  $(a_1^i,a_2^i)\in\R^2\setminus\{(0,0)\}$ and $b\in\R\setminus\{0\}$.
	\end{theorem}
	Theorem \ref{teoAutofunzioniAutovalori} is the counterpart in dimension $N=2$ of the results in \cite{GrossiPacellaMathZ}. \\
As clear from the proof of Theorem \ref{theorem:Morse1bubble}, to get \eqref{diesis} it is sufficient to know the precise asymptotic behavior of $\lambda_{i,n}$ for $i=2,3$. However we believe that it is interesting to see that also the eigenvalue $\lambda_{4,n}$ converges to $1$ and know exactly its rate of convergence, as well as the asymptotic limit of the corresponding eigenfunction which is essentially the Green function.

	\
	
	Finally we focus on the case when the domain $\Omega$ is convex.\\  
	In this situation using Theorem \ref{teoAsymptotic} and a result of \cite[Theorem 2.4]{GrossiTakahashi} we have that the solutions to \eqref{problem} cannot concentrate at more than one point, essentially because the Robin function has only one critical point, hence necessarily $k=1$ in  Theorem \ref{teoAsymptotic} and so Theorem \ref{theorem:Morse1bubble} applies. In particular, combining  Theorem \ref{theorem:Morse1bubble}  with the results in \cite{CaffarelliFriedman} about the critical points of the Robin function in convex planar domains,
	we get:
	\begin{corollary} 
		\label{teoMorse}
		Let $\Omega\subset\mathbb R^2$ be  a smooth bounded and convex domain. Let $(u_p)$ be a family of solutions which satisfies \eqref{energylimit}. Then there exists $p^{\star}=p^{\star}(\Omega, C)>1$ (where   $C$ is the constant in \eqref{energylimit})  such that 
		\begin{equation}\label{diesis2} m(u_p)=1\  \mbox{ and }\  u_p\mbox{ is nondegenerate,  if }p\geq p^{\star}.\end{equation}
	\end{corollary}
	Observe that once $m(u_p)=1$ is proved, then the  nondegeneracy follows by the results in \cite{Lin} for  $\Omega\subset\mathbb R^2$ convex. Anyway here we get the nondegeneracy independently.\\
	Note also that this  result holds for  any family of solutions $u_p$, without passing to a subsequence, differently from Theorem \ref{theorem:Morse1bubble}.
	\\
	\\
	In view of the uniqueness result of \cite{Lin} for solutions of Morse index $1$, from Corollary \ref{teoMorse} it immediately derives that, in any convex domain, \eqref{problem} admits only one solution for $p$ large, as long as \eqref{energylimit} holds. In order to get rid of the condition \eqref{energylimit} and get the full uniqueness result for large $p$, we observe that in star-shaped domains it is possible to see (applying Pohozaev identity) that the energy bound \eqref{energylimit} is equivalent to an  uniform $L^{\infty}$-bound:
	\begin{equation}\label{boundSirakov}\sup_p\|u_p\|_{L^{\infty}(\Omega)}\leq \widetilde C,\quad \mbox{for some constant }\widetilde C>0.
	\end{equation}
	The important bound \eqref{boundSirakov} has been indeed proved very recently in \cite{Sirakov}, therefore \emph{the result \eqref{diesis2} holds without assuming \eqref{energylimit}}. This 	together with the results  in \cite{Lin}  shows  that the Gidas-Ni-Nirenberg uniqueness conjecture in planar  convex domains is true for large values of the exponent  $p$.
	\begin{theorem}\label{teoUnicita}
		Let $\Omega\subset\mathbb R^2$ be a smooth bounded convex domain. Then there exists $p^{\star}=p^{\star}(\Omega)>1$ such that \eqref{problem} admits a unique solution for any $p\geq p^{\star}$.
	\end{theorem}

Note that we do not make any assumption on the solution (as in \cite{Lin}) or on the geometry of the convex domain $\Omega$ (as in \cite{DamascelliGrossiPacella, Dancer88,  GrossiADE2000, Computer, Computer2, Zou}). On the other side our result applies only
in dimension $N=2$ and for large values of the exponent $p$.

	\

	The paper is organized as follows: in Section \ref{section:preliminaries} we collect some notations and preliminary results. Sections \ref{section:ele12} and \ref{section:ele34} are devoted to the asymptotic analysis of the eigenvalues and eigenfunctions of the linearized operator which leads to the proof of Theorem \ref{teoAutofunzioniAutovalori}. At the end of Section \ref{section:ele34} we then deduce Theorem \ref{theorem:Morse1bubble}. Section \ref{section:conclusion} focuses on the case $\Omega$ convex and contains the proof of Corollary \ref{teoMorse}.
	\\
	\\	
	(*) The results of Theorem \ref{theorem:Morse1bubble}, Theorem \ref{teoAutofunzioniAutovalori} and Corollary \ref{teoMorse} were announced during two conferences in Brasilia (\cite{FilomenaTalk}, September 2017) and Roma (\cite{FrancescaTalk}, January 2018), where the question of obtaining the uniform estimate \eqref{boundSirakov} was posed and its validity was conjectured.

	\tableofcontents

\section{Preliminary results}\label{section:preliminaries}	

\subsection{Green and Robin functions}
Let $G(x,y)$ be the Green function of $-\Delta$ in $\Omega$ with Dirichlet boundary conditions. Then
\begin{equation}\label{HparteregolareGreen}
G(x,y)=-\frac1{2\pi}\log|x-y|-H(x,y),
\end{equation}
where $H(x,y)$ is the regular part of the Green function. Let $R(x)=H(x,x)$ be the Robin function of $\Omega$.

\begin{lemma}\label{thm:Robin}
	Let $\Omega\subset \R^2$ be a bounded domain, then the Robin function satisfies
	\begin{equation}\label{Robinequation}
	\Delta R>0\quad \mbox{ in }\Omega.
	\end{equation}
\end{lemma}
\begin{proof} 
	If  $\Omega$ is a simply connected domain one has (see e.g. \cite{Gustafsson})
	\[
	\Delta R=8\pi e^{2\pi R}\quad\mbox{ in }\Omega,
	\]
	so \eqref{Robinequation} immediately follows. More in general in \cite{BandleFlucher} it is proved that in any bounded domain 
	\[
	\Delta R=2K (x,x)\quad\mbox{ in }\Omega,
	\]
	where $K(x,y)$ is the Bergman kernel function. Then \eqref{Robinequation} follows immediately from the following characterization (see \cite{BergmanBOOK}):
	\[K(x,x)=\sup\{|f(x)|^2\ :\  f\in L^2(\Omega), \ \|f\|_{L^2(\Omega)}\leq 1, \ f \mbox{ is holomorfic in }\Omega\}.\]
\end{proof}
\begin{lemma} [\cite{CaffarelliFriedman}]\label{thm:Robin2}
	If $\Omega$ is any bounded convex domain, then $R$ is strictly convex and so it has a unique critical point which is a strict minimum. Moreover the corresponding Hessian matrix $D^2R$ in the point is positive definite.
\end{lemma}

Furthermore we list some computations derived in \cite{GladialiGrossiCPDE} and \cite{GladialiGrossi}, see in particular \cite[Lemma 3.4]{GladialiGrossi}.
\begin{lemma}
	For any $y\in\Omega$:
	\begin{equation}
	\label{Green1}
	\int_{\partial\Omega}(x-y)\cdot \nu(x) \left(\frac{\partial G}{\partial \nu}(x,y)\right)^2 d\sigma_x=\frac1{2\pi};
	\end{equation}
	\begin{equation}
	\label{Green2}
	2\int_{\partial\Omega} (x-y)\cdot \nu(x)\frac{\partial G}{\partial \nu}(x,y) d\sigma_x=\frac{\partial R}{\partial y_j}(y);
	\end{equation}
	\begin{equation}
	\label{Robin1}
	\int_{\partial\Omega}\nu_j(x) \left(\frac{\partial G}{\partial \nu}(x,y)\right)^2 d\sigma_x=\frac{\partial R}{\partial y_j}(y);
	\end{equation}
	\begin{equation}
	\label{Robin2}
	\int_{\partial\Omega} \frac{\partial G}{\partial x_j}(x,y) \frac{\partial}{\partial y_k}\left(\frac{\partial G}{\partial \nu}(x,y)\right) d\sigma_x=\frac12\frac{\partial^2 R}{\partial y_j \partial y_k}(y).
	\end{equation}
\end{lemma}

\

\subsection{Asymptotic behavior of multi-spike solutions}

The complete analysis of the asymptotic behavior as $p\rightarrow +\infty$ of the solutions to \eqref{problem}  has been done in \cite{DIPpositive} and refined in \cite{DGIPsqrte}. The results can be summarized in the following:

\begin{theorem}[\cite{DIPpositive, DGIPsqrte}]
	\label{teoAsymptotic} 
	Let $(u_p)$ be a family of solutions to \eqref{problem}  and assume that \eqref{energylimit} holds.
	Then there exist a finite number $k$   of distinct  points $x_{j,\infty}\in\Omega$, $j=1,\ldots, k$   and a sequence $p_n\rightarrow +\infty$ as $n\rightarrow +\infty$
	such that setting \[\mathcal S:=\{x_{1,\infty},\ldots, x_{k,\infty}\}\]
	one has
	\begin{equation}\label{pu_va_a_zeroTeo}
	\lim_{n\rightarrow +\infty}	\sqrt{p_n}u_{p_n}=0\ \mbox{ in $C^2_{loc}(\overline\Omega\setminus\mathcal{S})$;}
	\end{equation}
	\begin{equation}\label{convergenzapup}
	\lim_{n\rightarrow +\infty}p_nu_{p_n}(x)= 8\pi \sqrt{e} \sum_{j=1}^k G(x,x_{j,\infty})\  \mbox{ in } 
	C^2_{loc}(\overline\Omega\setminus\mathcal S), 
	\end{equation}
	the energy satisfies
	\begin{equation}
	\label{convergenzaenergia}
	\lim_{n\rightarrow +\infty}	p_n\int_\Omega |\nabla u_{p_n}(x)|^2\,dx= 8\pi e\cdot k
	\end{equation}
	and the concentration points $x_{j,\infty}, \ j=1,\ldots, k$ fulfill
	the system 
	\begin{equation}\label{x_j relazioneTeo}
	-\nabla_x H(x_{j,\infty},x_{j,\infty})+\sum_{j\neq h}\nabla_x G(x_{j,\infty},x_{h,\infty})=0.		
	\end{equation}
	Furthermore, for $r>0$  such that $B_{3r}(x_{j,\infty})\subset\Omega$, $\overline{B_{2r}(x_{j,\infty})}\cap \overline{B_{2r}(x_{h,\infty})}=\emptyset$, for any $j, h=1,\ldots, k$, $h\neq j$. Let $x_{j,n}\in\overline{B_{2r}(x_{j,\infty})}$	be the sequence defined as 
	\begin{equation}
	\label{ptimaxloc}
	u_{p_n}(x_{j,n})=\max_{\overline{B_{2r}(x_{j,\infty})}}u_{p_n}
	\end{equation}
	then, for any $j=1,\ldots, k$,
	\begin{equation}\label{convptimax}
	\lim_{n\rightarrow +\infty}x_{j,n}= x_{j,\infty},
	\end{equation}
	\begin{equation}\label{misqrte}
	\lim_{n\rightarrow +\infty} u_{p_n}(x_{j,n})=\sqrt{e},
	\end{equation}
	\begin{equation}
	\label{varep}
	\lim_{n\rightarrow +\infty}\varepsilon_{j,n}:=\left[ p_n u_{p_n}(x_{j,n})^{p_n-1}\right]^{-1/2}= 0\end{equation}
	and setting
	\begin{equation}
	\label{defRiscalataMax}
	w_{j,n}(y):=\fr{p_n}{u_{p_n}(x_{j,n})}(u_{p_n}(x_{j,n}+\varepsilon_{j,n} y)-u_{p_n}(x_{j,n})),\quad  y\in  \Omega_{j,n}:=\frac{\Omega-x_{j,n}}{\varepsilon_{j,n}},
	\end{equation}
	one has
	\begin{equation}\label{convRiscalateNeiMax}
	\lim_{n\rightarrow +\infty} w_{j,n}=U\ \mbox{ in }\ C^2_{loc}(\R^2)
	\end{equation}
	where
	\begin{equation}\label{definizioneU}
	U(x)=-2\log\left(1+\frac{ |x|^2}{8}\right)
	\end{equation}
	is a solution of the Liouville equation 
	\begin{equation}\label{LiouvilleEquationINTRO}
	\left\{
	\begin{array}{lr}
	-\Delta U=e^U\quad\mbox{ in }\R^2\\
	\int_{\R^2}e^Udx= 8\pi.
	\end{array}
	\right.
	\end{equation}
	Moreover there exists $C>0$ such that:
	\begin{equation}\label{P3}
	p_n\min_{j=1,\ldots,k}|x-x_{j,n}|^2|u_{p_n}(x)|^{p_n-1}\leq C\qquad\mbox{ for any } x\in\Omega \mbox{ and for any }n
	\end{equation}
	and
	\begin{equation}\label{P4}
	p_n\min_{j=1,\ldots,k}|x-x_{j,n}||\nabla u_{p_n}(x)|\leq C\qquad\mbox{for any }x\in\Omega\mbox{ and for any }n.
	\end{equation}
\end{theorem}

\

\begin{remark}	
		It is easy to see that \eqref{energylimit} implies 
		\begin{equation}
		\label{boundLoo}
		\|u_p\|_{L^{\infty}(\Omega)}\leq \widetilde C,
		\end{equation}
		for some uniform constant $\widetilde C>0$. Besides from \eqref{pu_va_a_zeroTeo} and \eqref{misqrte} we have that for the whole family $u_p$
		\begin{equation}
		\label{sqrte}
		\lim_{p\rightarrow +\infty}\|u_p\|_{L^{\infty}(\Omega)}=\sqrt{e}.
		\end{equation}
\end{remark}
Next we collect other useful properties from \cite{DIPpositive}. Throughout the section  $(u_p)$ is a family of solutions to \eqref{problem} which satisfies the uniform bound \eqref{energylimit} and we keep the notations of Theorem \ref{teoAsymptotic}.   
\begin{lemma}
	For any $\alpha\in\R$ and any compact set $\omega\subset\Omega\setminus \mathcal S$
	\begin{equation}
	\label{utile}
	\lim_{n\rightarrow +\infty} p_n^{\alpha}\|u_{p_n}\|^{p_n-1}_{L^\infty(\omega)}=0.
	\end{equation}
	Moreover there exists $C>0$ such that
	\begin{equation}
	\label{bound_int_p-1}
	p\int_\Omega u_p^{p-1} dx\leq C.
	\end{equation}
\end{lemma}
\begin{proof}
	\eqref{utile} follows immediately from \eqref{pu_va_a_zeroTeo} in Theorem \ref{teoAsymptotic}.
	\\Moreover
	\begin{eqnarray}
	p\int_\Omega u_p^{p-1} dx&\overset{\mbox{\footnotesize{H\"older}}}{\leq} & p|\Omega|^{\frac{2}{p+1}}\left(\int_\Omega u_p^{p+1}dx\right)^{\frac{p-1}{p+1}}\overset{\eqref{problem}}{\leq}  p \left(\int_\Omega |\nabla u_p|^2 dx\right)^{\frac{p-1}{p+1}}(1+o(1))\nonumber\\
	&\leq&p \left(\int_\Omega |\nabla u_p|^2 dx\right)(1+o(1))\overset{\eqref{energylimit}}{\leq} C.
	\end{eqnarray} 
\end{proof}

Finally we give pointwise decay estimates for $w_{j,n}$ which will be crucial to pass to the limit in various integral identities (see Remark \ref{rem:lebesgue} below).
\begin{lemma} \label{lemma:lebesgue}
	For any $\gamma>0$ there exist $R_{\gamma}>0$, $\widetilde C_{\gamma}$, $C_{\gamma}>0$ and $n_\gamma\in\N$ such that for any $n\geq n_\gamma$
	\begin{equation}
	\label{prelebesgue}
	w_{j,n}(z)\leq \left(4-\frac{\gamma}{2}\right)\log\frac{1}{|z|}+\widetilde C_\gamma\qquad\mbox{for }R_\gamma\leq |z|\leq\frac{r}{\varepsilon_{j,p}}
	\end{equation}
	and
\begin{equation}
\label{lebesgue}
	0\leq \left(1+\frac{w_{j,n}(z)}{p_n}\right)^{p_n-1}\leq 
	\frac{ C_\gamma}{1+|z|^{4-\gamma}}\qquad \mbox{for }|z|\leq \frac{r}{\varepsilon_{j,n}}.
\end{equation} 
\end{lemma}
\begin{proof}
\eqref{prelebesgue} follows directly from \cite[Lemma 4.4 and Proposition 4.3]{DIPpositive}.
\\
Next we derive \eqref{lebesgue}. Observe that by \eqref{convptimax} $B_r(x_{j,n})\subset B_{2r}(x_{j,\infty})$ for $n$ sufficiently large. As a consequence, by \eqref{ptimaxloc}, $w_{j,n}\leq0$ in $B_{\frac{r}{\varepsilon_{j,n}}}(0)\subset\Omega_{j,n}$ for $n$ large, which implies
\begin{equation}\label{ven}
\left(1+\frac{w_{j,n}(z)}{p_n}\right)^{p_n-1}\leq 1\qquad\mbox{for any $z\in B_{\frac{r}{\varepsilon_{j,n}}}(0)$}.
\end{equation}
Moreover, since  for $n$ sufficiently large $\frac{p_n-1}{p_n}\geq \frac{4-\gamma}{4-\frac\gamma2}$,  by \eqref{prelebesgue} we get
\begin{equation}\label{ven2}
\left(1+\frac{w_{j,n}(z)}{p_n}\right)^{p_n-1}=e^{(p_n-1)\log\left(1+\frac{w_{j,n}(z)}{p_n}\right)}\leq e^{\frac{p_n-1}{p_n} w_{j,n}(z)}\overset{w_{j,n}\leq 0}{\leq} e^{\frac{4-\gamma}{4-\frac{\gamma}{2}} w_{j,n}(z)}\overset{\eqref{prelebesgue}}{\leq}\frac{  \widehat C_{\gamma}}{|z|^{4-\gamma}}, 
\end{equation}
for $R_{\gamma}\leq 
	|z|\leq \frac{r}{\varepsilon_{j,n}}$ and for some $\widehat C_\gamma>0$. Combining \eqref{ven2} with \eqref{ven} we get the thesis.
\end{proof}

\

\begin{remark}
	\label{rem:lebesgue} 
By \eqref{convRiscalateNeiMax} and Lemma \ref{lemma:lebesgue} we can use the dominated convergence theorem and pass to limit proving that
	\[
\lim_{n\rightarrow +\infty}	\int_{D_n\cap B_{\frac{r}{\varepsilon_{j,n}}}(0)}\left(1+\frac{w_{j,n}(z)}{p_n}\right)^{p_n-1}g_n(z)\,dz
	=\int_{\R^2}e^{U(z)}g(z)\,dz,
	\]
	for any sequence of sets $D_n\subset\R^2$, $D_n\to \R^2$ and any sequence of  functions $g_n$ such that $g_n\to g$ pointwise in $\R^2$ as $n\rightarrow +\infty$ and $|g_n(z)|\leq h(z)$ for $|z|\leq\frac{r}{\varepsilon_{j,n}}$, where the function $h$ is such that $\frac{h(z)}{1+|z|^{4-\gamma}}\in L^1(\R^2)$, for some $\gamma >0$.
\end{remark}

\

\

{\bf Notation in the case $\mathbf{k=1}$.}
In this paper we focus on  solutions to \eqref{problem} for which $k=1$. In order to simplify the notation let us set:\\

\begin{itemize}
\item $x_\infty:=x_{1,\infty}$, which, by \eqref{x_j relazioneTeo}, is a critical point of the Robin function $R(x):=H(x,x)$;
\item $x_n:=x_{1,n}$ the local maximum point in \eqref{ptimaxloc}, which coincides now with the global maximum, namely
\begin{equation}\label{ptimaxloc1}
u_{p_n}(x_n)=\|u_{p_n}\|_{L^{\infty}(\Omega)};
\end{equation}
\item
$\varepsilon_n:=\varepsilon_{1,n}$  the parameter in \eqref{varep}, namely 
	\begin{equation}
\label{varep1}
\varepsilon_{n}:=\left[ p_n \|u_{p_n}\|_{L^{\infty}(\Omega)}^{p_n-1}\right]^{-1/2}(\rightarrow  0 \mbox{ as }n\rightarrow +\infty);\end{equation}
\item
$w_{n}:=w_{1,n}$ the rescaled function in \eqref{defRiscalataMax};
\item
$\Omega_n:=\Omega_{1,n}$ the rescaled domain in \eqref{defRiscalataMax}.
\end{itemize}

\

We will also use the following notation for this rescaling of $u_{p_n}$:
\begin{equation}\label{utilde}
\widetilde u_{p_n}(y)=u_{p_n}(x_n+\varepsilon_n y),\qquad\quad 	 \mbox{for $y\in\Omega_n$}.
\end{equation} 
Observe that 
	\begin{equation}
\label{boundLootilde}
\|\widetilde u_p\|_{L^{\infty}(\Omega)}=\|u_p\|_{L^{\infty}(\Omega)}\overset{\eqref{boundLoo}}{\leq} \widetilde C,
\end{equation} 
moreover, by \eqref{convRiscalateNeiMax} and \eqref{sqrte} 
\begin{equation}\label{convutilde}
\widetilde u_{p_n}=\|u_{p_n}\|_{L^{\infty}(\Omega)}\left(1+\frac{w_n}{p_n} \right)\longrightarrow \sqrt{e}\quad \mbox{ in }C^0_{loc}(\R^2) \mbox{, as }n\rightarrow +\infty,
\end{equation}
\begin{equation}\label{convnablautilde}
p_n\nabla \widetilde u_{p_n}=\|u_{p_n}\|_{L^{\infty}(\Omega)}\nabla w_n\longrightarrow \sqrt{e}\,\nabla U\quad \mbox{ in }C^0_{loc}(\R^2) \mbox{, as }n\rightarrow +\infty,
\end{equation}
Moreover \eqref{P4} becomes
	\begin{equation}\label{P4tilde}
p_n|y||\nabla \widetilde u_{p_n}(y)|\leq C\qquad\mbox{for any }y\in\Omega_n\mbox{ and for any }n.
\end{equation}

\

\

\subsection{Linearized equation at $1$-spike solutions: preliminary asymptotic results}
Let $(u_p)$ be a family of solutions to \eqref{problem}, satisfying the uniform bound \eqref{energylimit}, let $p_n$ be the sequence in Theorem \ref{teoAsymptotic} and assume that Theorem \ref{teoAsymptotic} holds with $k=1$. \\
Let us denote respectively by $\lambda_{i,n}$ and $v_{i,n}$, $i\in\N$,  the eigenvalues (counted with multiplicity) and the associated  eigenfunctions  of  the linearized problem
\begin{equation}
\label{eigenvalueProblem}
\left\{
\begin{array}{ll}
-\Delta v=\lambda p_nu_{p_n}^{p_n-1} v & \mbox{in }\Omega\\
v=0 & \mbox{on }\partial\Omega\\
\|v\|_{L^{\infty}(\Omega)}=1.
\end{array}
\right.
\end{equation}
We may assume that the eigenfunctions are orthogonal in the space $H^1_0(\Omega)$, i.e. 
\begin{equation}\label{autofunzOrtogonali}
\int_{\Omega}\nabla v_{i,n}\nabla v_{j,n}=0 \quad \forall j\neq i. 
\end{equation}
Taking $\lambda=\frac{1}{p_n}$ in \eqref{eigenvalueProblem} we have that $v_{1,n}=u_{p_n}/\|u_{p_n}\|_{L^{\infty}(\Omega)}$ is a corresponding eigenfunction and hence $\frac{1}{p_n}=\lambda_{1,n}$ is the first eigenvalue.
\\
\\
The following result holds:
\begin{lemma}
	For any eigenfunction $v_{i,n}$, $i\in\N$  and for any $y\in\R^2$ we have the following integral identities:
	\begin{equation}\label{integralidentity}
	\int_{\partial\Omega}(x-y)\cdot\nabla u_{p_n}\frac{\partial v_{i,n}}{\partial\nu}d\sigma_x=(1-\lambda_{i,n})p_n\int_\Omega  u_{p_n}^{{p_n}-1}v_{i,n}\left((x-y)\cdot\nabla u_{p_n}+\frac{2}{{p_n}-1}u_{p_n}\right)dx.
	\end{equation}
	\begin{equation}\label{star}
	\int_{\partial\Omega}\frac{\partial u_{p_n}}{\partial x_j}\frac{\partial v_{i,n}}{\partial\nu}d\sigma_x=(1-\lambda_{i,n})p_n\int_\Omega  u_{p_n}^{{p_n}-1}\frac{\partial u_{p_n}}{\partial x_j}v_{i,n}dx.
	\end{equation}
\end{lemma}
\begin{proof}
	\eqref{integralidentity} follows arguing exactly as in \cite[Lemma 4.3]{GrossiPacellaMathZ}, 
	\eqref{star} can be proved as in \cite[Lemma 5.1]{GrossiPacellaMathZ}.	
\end{proof}
Our aim is to study the asymptotic behavior of the eigenfunctions and eigenvalues  $v_{i,n}$ and $\lambda_{i,n}$, $i\in\N$, as $n\rightarrow +\infty$. It is convenient to rescale the eigenfunctions $v_{i,n}$ as follows:
\begin{equation}\label{rescaledeigenfunct}
\widetilde{v}_{i,n}(x):=  v_{i,n}(x_n+\varepsilon_n x), \ \ x\in {\Omega}_n,
\end{equation}
where $x_n$ and $\varepsilon_n$ are as in \eqref{ptimaxloc1} and \eqref{varep1} respectively. 
Then, it is easy to see that
$(\lambda_{i,n},\widetilde{v}_{i,n})$ are the eigenpairs for the following eigenvalue problem
\begin{equation}\label{eqrescaledeigenfunct}
\left\{\begin{array}{lr}
-\Delta v= \lambda V_n(x)v& \mbox{ in }{\Omega}_n\\
v=0 &\mbox{ on }\partial {\Omega}_n\\
\|v\|_{L^{\infty}({\Omega}_n)}=1
\end{array}\right.
\end{equation}
where
\begin{equation}\label{defVn}
V_n(x):=\left(\frac{u_{p_n}(x_n+\varepsilon_n x)}{u_{p_n}(x_n)} \right)^{p_n-1}=\left(  1+\frac{w_{n}(x)}{p_n}  \right)^{p_n-1},
\end{equation}
and  $w_n=w_{1,n}$ is  the rescaled function defined  in  \eqref{defRiscalataMax}.

\

In the rest of the section we prove some crucial intermediate asymptotic results  for eigenvalues and eigenfunctions which will be used throughout the paper.
\begin{lemma}\label{lemma:varphineq0} Let $\widetilde{v}_{i,n}$, $i\in\N$, be the rescaled eigenfunction defined in \eqref{rescaledeigenfunct}. 
If  $\widetilde{v}_{i,n}\to \widetilde v $ in $C^0_{loc}(\R^2)$ and $\lambda_{i,n}\rightarrow \Lambda\in[0,+\infty)$ as $n\rightarrow +\infty$, 
	then $\widetilde v\not\equiv0$.
\end{lemma}

\begin{proof}	
Let us first consider the eigenfunction  $v_{i,n}$, which solves \eqref{eigenvalueProblem} with $\lambda=\lambda_{i,n}$.\\
Observe that, by  \eqref{utile},  $\lambda_{i,n}p_n u_{p_n}^{p_n-1}v_{i,n}\to 0$ locally uniformly in $\Omega\setminus\{x_\infty\}$ as $n\rightarrow +\infty$. Hence from \eqref{eigenvalueProblem}, by standard elliptic regularity estimates,   we deduce that
\begin{equation}
\label{ven4}
v_{i,n}\longrightarrow0\qquad\mbox{locally uniformly in }\overline\Omega\setminus \{x_\infty\} \mbox{, as }n\rightarrow +\infty.
\end{equation}
Let now  $\widetilde v_{i,n} $ be the rescaled eigenfunction. Let us assume without loss of generality that $\max_{\Omega_n} \widetilde v_{i,n}=1$ and let us denote by $s_n\in\Omega_n$ a point such that 
\begin{equation}\label{tildevin=1}
\widetilde v_{i,n}(s_n)=1.
\end{equation}
Observe that $B_{\frac{r}{\varepsilon_n}}(0)\subset\Omega_n$ (by the choice of $r$ in Theorem \ref{teoAsymptotic}) and that
\begin{equation}
\label{snstainpalla}
|s_n|<\frac{r}{\varepsilon_n}\quad \mbox{ for $n$ large,}
\end{equation}
indeed $v_{i,n}(x_n+\varepsilon_n s_n)=1$ by \eqref{rescaledeigenfunct} and \eqref{tildevin=1}, so that by \eqref{ven4} and \eqref{convptimax} we can easily deduce that $x_n+\varepsilon_n s_n\in B_r(x_n)$ for $n$ large, namely \eqref{snstainpalla}.\\
Assume by contradiction that $\widetilde v\equiv0$, namely  that 
\begin{equation}\label{vtildena0}
\widetilde v_{i,n}\longrightarrow0\quad\mbox{locally uniformly in $\R^2$, as $n\rightarrow +\infty$.}
\end{equation}
Then necessarily $s_n\to+\infty$ and so in particular 
\begin{equation}\label{snmaggiorediuno}
|s_n|>1 \quad \mbox{for $n$ large.}
\end{equation}
Let $z_n$ be the Kelvin transform of $\widetilde v_{i,n}$, namely
\[
z_{n}(x):=\widetilde v_{i,n}\left(\frac{x}{|x|^2}\right). 
\]
Observe that $z_n$ is well defined in $\R^2\setminus B_{\frac{\varepsilon_n}{r}}(0)$ (since $ B_{\frac{r}{\varepsilon_n}}(0)\subset\Omega_n$ and $\widetilde v_{i,n}$ is defined in $\Omega_n$) and by \eqref{eqrescaledeigenfunct}, it  satisfies
\begin{equation}\label{equazionezn}
	-\Delta z_{n}=\frac{\lambda_{i,n}}{|x|^4}V_n(\frac{x}{|x|^2})z_{n} \qquad\mbox{in }\R^2\setminus B_{\frac{\varepsilon_n}{r}}(0).
\end{equation}
	Moreover, by \eqref{vtildena0},   
	\begin{equation}\label{znconvZeroPunt}z_{n}(x)\to 0\mbox{ as }n\rightarrow +\infty,\mbox{ pointwise for any }x\neq 0.\end{equation}
Let us define
\[f_n(x):=\left\{
\begin{array}{ll}
\displaystyle\frac{\lambda_{i,n}}{|x|^4}V_n(\frac{x}{|x|^2})z_{n}(x) &\mbox{ for  }x\in B_1(0)\setminus \overline{B_{\frac{\varepsilon_n}{r}}(0)}
\\
&\\
0&\mbox{ for }x\in B_{\frac{\varepsilon_n}{r}}(0).
\end{array}\right.\]
By \eqref{znconvZeroPunt}, the definition of $V_n$ in \eqref{defVn} and using \eqref{convRiscalateNeiMax} we have that
$f_n\to0$ pointwise in $B_1(0)$,  as $n\rightarrow +\infty$. Moreover by the estimate \eqref{lebesgue}  one has 
\[
f^2_n(x) 
\leq h(x):=
\frac{(\Lambda+1)^2C_\gamma^2}{|x|^{2\gamma}(1+|x|^{4-\gamma})^2} \quad
\in L^1(B_1(0)) \mbox{, choosing $\gamma<\frac12$}.
\]
Thus by the dominated convergence theorem 
\begin{equation}\label{convInL2}
\lim_{n\rightarrow +\infty}\|f_n\|_{L^2(B_1(0)}=0.
\end{equation}
As a consequence, considering $g_n\in H^1_0(B_1(0))$ such that 
\begin{equation}\label{equaziongn}
\left\{
\begin{array}{ll}
-\Delta g_n=f_n & \mbox{ in }B_1(0)
\\
g_n=0 &\mbox{ on }\partial B_1(0),
\end{array}
\right.
\end{equation}
 by \eqref{convInL2} and  the elliptic regularity we have that 
\begin{equation} 
\label{ven3}
g_n\to0\quad\mbox{ uniformly in } B_1(0) \mbox{ as }n\rightarrow +\infty.
\end{equation}
Now we consider the difference $z_{n}-g_n$. This function is harmonic on $B_1(0)\setminus \overline{B_{\frac{\varepsilon_{n}}{r}}(0)}$  by the equations \eqref{equazionezn} and \eqref{equaziongn} and the maximum principle for harmonic functions guarantees
\begin{eqnarray*}
\|z_{n}-g_n\|_{L^\infty(B_1(0))\setminus \overline{B_{\frac{\varepsilon_{n}}{r}}(0)})}&\leq& \|z_{n} -g_n \|_{L^\infty(\partial B_1(0))}+\|z_{n} -g_n \|_{L^\infty(\partial {B_{\frac{\varepsilon_{n}}{r}}(0)})}\\
&\leq& \|z_{n}\|_{L^\infty(\partial B_1(0))}+\| z_{n}\|_{L^\infty(\partial {B_{\frac{\varepsilon_{n}}{r}}(0)})}+\| g_n\|_{L^\infty(\partial {B_{\frac{\varepsilon_{n}}{r}}(0)})}\\
&\overset{\eqref{ven3}}{=}&  \|\widetilde v_{i,n}\|_{L^\infty(\partial B_1(0))}+\|\widetilde v_{i,n}\|_{L^\infty(\partial B_{\frac{r}{\varepsilon_n}}(0))}+o(1)\\
&\overset{\eqref{vtildena0}}{=}& \|v_{i,n}\|_{L^\infty(\partial B_r(x_n))}+o(1)
\\
&\overset{\eqref{convptimax}}{=}& \|v_{i,n}\|_{L^\infty(\partial B_r(x_{\infty}))}+o(1)\overset{\eqref{ven4}}{=}o(1).
\end{eqnarray*}
In conclusion, recalling again \eqref{ven3}, we get
\[
\|z_{n}\|_{L^\infty(B_1(0)\setminus \overline{B_{\frac{\varepsilon_{n}}{r}}(0)})}\leq \|g_n\|_{L^\infty(B_1(0))}+\|z_{n}-g_n\|_{L^\infty(B_1(0)\setminus \overline{B_{\frac{\varepsilon_{n}}{r}}(0)})}=o(1),
\]
which contradicts the fact that by \eqref{snstainpalla} and \eqref{snmaggiorediuno}
\[\frac{s_n}{|s_n|^2}\in B_1(0)\setminus \overline{B_{\frac{\varepsilon_{n}}{r}}(0)}\]
and
\[
z_{n}\left(\frac{s_n}{|s_n|^2}\right)=\widetilde v_{i,n}(s_n)\overset{\eqref{tildevin=1}}{=}1.
\]
\end{proof}

\

Next, let us recall a well known characterization of the kernel of the linearized operator at $U$ of the Liouville equation obtained in \cite{ElMehdiGrossi}.
\begin{lemma}\label{thm:autofunznucleo}
Let $v\in C^2(\R^2)$ be a solution of the following problem
\begin{equation}\label{quadratino}
\left\{
\begin{array}{ll}
-\Delta v=e^U v &\mbox{in $\R^2$}\\
v\in L^\infty(\R^2),
\end{array}
\right.
\end{equation}
where $U$ is defined in \eqref{definizioneU}. 
Then
\[
v(y)=\sum_{k=1}^2\frac{a_k y_k}{8+|y|^2}+b\frac{8-|y|^2}{8+|y|^2}
\]
for some $a_k$, $b\in\R$.
\end{lemma}

\begin{lemma}\label{lemma:limab}
Let $i\in\N$. If $\lambda_{i,n}\to1$, as $n\to+\infty$,  then  there exists   $\R^3\ni (a^i_1,a^i_2,b^i)\neq(0,0,0)$ such that
\begin{equation}\label{tesiLemmalimab}
\widetilde{v}_{i,n}(y)\longrightarrow\sum_{j=1}^2\frac{a^i_j y_j}{8+|y|^2}+ b^i\frac{8-|y|^2}{8+|y|^2}\quad \mbox{ as } n\to+\infty \mbox{ in }C^1_{loc}(\R^2),
\end{equation}
where $\widetilde{v}_{i,n}$ is the rescaled eigenfunction defined in \eqref{rescaledeigenfunct}.
\end{lemma}
\begin{proof}
The rescaled eigenfunction $\widetilde{v}_{i,n}$ satisfies the eigenvalue problem \eqref{eqrescaledeigenfunct}.
By the assumption $\lambda_{i,n}\to1$, \eqref{convRiscalateNeiMax} and standard elliptic estimates we have that $\widetilde v_{i,n}$ converges in $C^1_{loc}(\R^2)$ to a solution $\widetilde v_i$  of \eqref{quadratino}. By Lemma \ref{thm:autofunznucleo} we have that 
\[
\widetilde v_i(y)=\sum_{k=1}^2\frac{a_k^i y_k}{8+|y|^2}+b^i\frac{8-|y|^2}{8+|y|^2},
\]
for some $(a^i_1,a^i_2,b^i )\in\R^3$. 
At last  $(a^i_1,a^i_2,b^i )\neq (0,0,0)$ by Lemma \ref{lemma:varphineq0} (applied with $\Lambda=1$).
\end{proof}

\begin{proposition}\label{prop:assurdob}
Let $i\in\N$. Suppose that  $\lambda_{i,n}\to1$ as $n\to+\infty$ and that
$b^i\neq0$, where $b^i$ is the constant in \eqref{tesiLemmalimab}. 
Then:	
\begin{equation}\label{riorganizzazione1}
p_n v_{i,n}\to-8\pi b^i G(x,x_{\infty})\qquad \mbox{in $C^1_{loc}(\overline\Omega\setminus\{x_\infty\})$}
\end{equation}
and 
\begin{equation}\label{riorganizzazione2}
\lambda_{i,n}=1+\frac{6}{p_n}(1+o(1))\qquad\mbox{as $n\to+\infty$}.
\end{equation}
\end{proposition}
\begin{proof}
	{\emph{Step 1. We prove \eqref{riorganizzazione1}.}}\\
 Multiplying equations \eqref{problem} (with $p=p_n$) and \eqref{eigenvalueProblem} (with $p=p_n$ and $v=v_{i,n}$) by $v_{i,n}$ and $u_{p_n}$ respectively, integrating by parts and subtracting, we get
\[
\int_\Omega u_{p_n}^{p_n} v_{i,n}dx=0.
\]
Then
\begin{eqnarray}\label{ali}
0=\int_\Omega u_{p_n}^{p_n}v_{i,n} dx&=&\int_\Omega u_{p_n}^{{p_n}-1}(u_{p_n}-\|u_{p_n}\|_{L^{\infty}(\Omega)}+\|u_{p_n}\|_{L^{\infty}(\Omega)})v_{i,n}dx
\nonumber\\
&=&  \|u_{p_n}\|_{L^{\infty}(\Omega)}\int_\Omega u_{p_n}^{{p_n}-1}v_{i,n}dx +\int_\Omega u_{p_n}^{{p_n}-1}(u_{p_n}-\|u_{p_n}\|_{L^{\infty}(\Omega)})v_{i,n}dx. 
\end{eqnarray}
Moreover
\begin{eqnarray}\label{baba}
\int_{\Omega\setminus B_r(x_n)}u_{p_n}^{{p_n}-1}(\|u_{p_n}\|_{L^{\infty}(\Omega)}-u_{p_n})|v_{i,n}|dx &\leq& 2 |\Omega|
\|u_{p_n}\|^{p_n-1}_{L^\infty(\Omega\setminus B_r(x_n))}\|u_{p_n}\|_{L^{\infty}(\Omega)}
\nonumber\\
&\overset{(\star)}{\leq}& 2 |\Omega|
\|u_{p_n}\|^{p_n-1}_{L^\infty(\Omega\setminus B_{\frac{r}{2}}(x_\infty))}\|u_{p_n}\|_{L^{\infty}(\Omega)}
\nonumber\\
&\overset{\eqref{utile}+\eqref{boundLoo}}{=}& o(1),
\end{eqnarray}
where in $(\star)$ we have used that $B_{\frac{r}{2}}(x_{\infty})\subseteq B_r(x_n)$ for $n$ large, which is a consequence of \eqref{convptimax}. 
So by \eqref{ali} and \eqref{baba} we have
\begin{equation}\label{ondina}
-\|u_{p_n}\|_{L^{\infty}(\Omega)}\int_\Omega u_{p_n}^{{p_n}-1}v_{i,n}dx=\int_{B_r(x_n)} u_{p_n}^{{p_n}-1}(u_{p_n}-\|u_{p_n}\|_{L^{\infty}(\Omega)})v_{i,n}dx+o(1).
\end{equation}
Next, rescaling and recalling the definitions of $\widetilde v_{i,n}$  (see \eqref{rescaledeigenfunct}) and  $w_{n}$ (see \eqref{defRiscalataMax}), we have
\begin{eqnarray}\label{Max}
&&\int_{B_r(x_n)} u_{p_n}^{{p_n}-1}(u_{p_n}-\|u_{p_n}\|_{L^{\infty}(\Omega)})v_{i,n}dx =
\nonumber\\
&\overset{\eqref{rescaledeigenfunct}}{=}& \varepsilon_{n}^2\int_{B_{\frac{r}{\varepsilon_{n}}}(0)}u_{p_n}^{{p_n}-1}(x_n+\varepsilon_n y)(u_{p_n}(x_n+\varepsilon_n y)-\|u_{p_n}\|_{L^{\infty}(\Omega)}) \widetilde{v}_{i,n}(y)dy \nonumber\\
&\overset{\eqref{varep1} + \eqref{defRiscalataMax}}{=}&\frac{\|u_{p_n}\|_{L^{\infty}(\Omega)}}{{p_n}^2}\int_{B_{\frac{r}{\varepsilon_n}}(0)}\left(1+\frac{w_{n}(y)}{p_n}\right)^{{p_n}-1} w_{n}(y)\widetilde{v}_{i,n}(y)dy
\nonumber\\
&\overset{\eqref{convRiscalateNeiMax}+ \eqref{tesiLemmalimab}}{=}&\frac{\|u_{p_n}\|_{L^{\infty}(\Omega)}}{{p_n}^2}\left(\int_{\R^2}e^{U(y)}U(y)\left(\sum_{j=1}^2\frac{a^i_j y_j}{8+|y|^2}+ b^i\frac{8-|y|^2}{8+|y|^2}\right)dy + o(1)\right),
\end{eqnarray}
where the passage to the limit in the last equality can be done   arguing  as in Remark \ref{rem:lebesgue}. Indeed, setting  $g_n:= w_n\widetilde v_{ i,n}$,
by \eqref{convRiscalateNeiMax} and \eqref{tesiLemmalimab} one has $g(y):= U (y)\left(\sum_{j=1}^2\frac{a^i_j y_j}{8+|y|^2}+ b^i\frac{8-|y|^2}{8+|y|^2}\right)$.
Moreover, by the local uniform convergence \eqref{convRiscalateNeiMax} of $w_n$  and its estimate \eqref{prelebesgue}, one can take 
\[h(y):=
\left\{
\begin{array}{lr}
C &\quad\mbox{ for }|y|\leq R_{\gamma}\\
\left(4-\frac{\gamma}{2}\right)|\log|y||+\widetilde C_{\gamma} &\quad \mbox{ for } R_{\gamma}<|y|\leq\frac{r}{\varepsilon_n}
\end{array}
\right.\]
and finally choosing $\gamma=2$,  Remark \ref{rem:lebesgue} applies.
%
%
%
%
%
%
%
%
%
\\
Substituting \eqref{Max} into \eqref{ondina}, we get 
\begin{eqnarray}\label{ago}
p_n^2\int_\Omega u_{p_n}^{{p_n}-1}v_{i,n} dx&=&\int_{\R^2}e^{U(y)}U(y)\left(\sum_{j=1}^2\frac{a^i_j y_j}{8+|y|^2}+ b^i\frac{8-|y|^2}{8+|y|^2}\right)dy + o(1)
\nonumber
\\
&\overset{\eqref{definizioneU}}{=}&-8\pi b^i+o(1).
\end{eqnarray}
As a consequence, using the Green's representation formula, for $x\neq x_\infty$, we can write
\begin{eqnarray*}
	{p_n} v_{i,n}(x)&=&{p_n}\lambda_{i,n}\int_\Omega G(x,y) {p_n} u_{p_n}^{{p_n}-1}(y)v_{i,n}(y) dy\\
	&=&  \lambda_{i,n} G(x,x_n)p_n^2\int_\Omega u_{p_n}^{{p_n}-1}(y)v_{i,n}(y) dy+\\
	& &+\underbrace{p_n^2 \lambda_{i,n} \int_\Omega (G(x,y)-G(x,x_n)) u_{p_n}^{{p_n}-1}(y)v_{i,n}(y) dy}_{=:I_{i,n}(x)}\\
	&\stackrel{\eqref{ago}}{=}&-8\pi b^i G(x,x_n)(1+o(1))+I_{i,n}(x)
	\\
	&\stackrel{\eqref{convptimax}}{=}&-8\pi b^i G(x,x_\infty)(1+o(1))+I_{i,n}(x).
\end{eqnarray*} 
To get \eqref{riorganizzazione1} it is enough to show that $I_{i,n}(x)=o(1)$ in $C^1_{loc}(\overline\Omega\setminus\{x_{\infty}\})$.\\
For any $x\in\overline\Omega\setminus\{x_\infty\}$ we can chose $\delta\in(0,r)$ (where $r$ is as in the statement of Theorem \ref{teoAsymptotic}) such that $x\notin B_{2\delta}(x_\infty)\subset\Omega$ and we can split $\Omega$ in three pieces: $\Omega\setminus B_\delta(x_n)$, $B_\delta(x_n)\setminus B_{p_n^{-2}}(x_n)$, $B_{p_n^{-2}}(x_n)$. Integrating separately in the three regions we obtain:
\begin{eqnarray}\label{set1}
&& p_n^2\lambda_{i,n}\int_{\Omega\setminus B_\delta(x_n)}|G(x,y)-G(x,x_n)| u_{p_n}^{{p_n}-1}(y)|v_{i,n}(y)| dy
\nonumber
\\
&&\qquad\qquad\stackrel{|y-x_n|
	\geq \delta}{\leq} p_n^2 C \|u_{p_n}\|^{{p_n}-1}_{L^\infty(\Omega\setminus B_{\tfrac{\delta}{2}}(x_\infty))}\overset{\eqref{utile}}{=}o(1).
\end{eqnarray}
Next, by scaling and applying \eqref{lebesgue} with $\gamma=1$, we get:
\begin{eqnarray}\label{set2}
&&\!\!\!\!\!\!\! \!\!\!\!\!\!\! \!\!\!\!\!\!\! \!\!\!\!\!\!\! \!\!\!\!\!\!\! p_n^2\lambda_{i,n}\int_{B_\delta(x_n)\setminus B_{\tfrac{1}{p_n^2}}(x_n)}|G(x,y)-G(x,x_n)| u_{p_n}^{{p_n}-1}(y)|v_{i,n}(y)| dy=
\nonumber\\
&\leq&
{p_n}C\int_{B_{\tfrac \delta{\varepsilon_n}}(0)\setminus B_{\tfrac{1}{p_n^2\varepsilon_n}}(0)}(|G(x,x_n+\varepsilon_n z)|+|G(x,x_n)|) \left(1+\frac{w_n(z)}{p_n}\right)^{{p_n}-1} dz
\nonumber\\
&\overset{\eqref{lebesgue}}{\leq}&
p_n\int_{B_{\frac{\delta}{\varepsilon_n}}(0)\setminus B_{\frac{1}{p_n^2\varepsilon_n}}(0)}(|G(x,x_n+\varepsilon_n z)|+|G(x,x_n)|)\frac{C}{1+|z|^{3}}dz
\nonumber\\
&\leq&
p_n\int_{B_{\frac{\delta}{\varepsilon_n}}(0)\setminus B_{\frac{1}{p_n^2\varepsilon_n}}(0)}(|G(x,x_n+\varepsilon_n z)|+|G(x,x_n)|)\frac{C}{|z|^{3}}dz
\nonumber\\
&\leq& 
p_n^7\varepsilon_n^3C\int_{B_{\frac{\delta}{\varepsilon_n}}(0)\setminus B_{\frac{1}{p_n^2\varepsilon_n}}(0)}(|G(x,x_n+\varepsilon_n z)|+|G(x,x_n)|)dz
\nonumber\\
&\leq& 
p_n^7 \varepsilon_n C\int_{B_{\delta}(x_n)}(|G(x,y)|+|G(x,x_n)|)dy
\nonumber\\
&\overset{\eqref{varep1}}{=}& \frac{ C p_n^{\frac{13}{2}}}{\|u_{p_n}\|_{L^{\infty}(\Omega)}^{\frac12 (p_n-1)}} \int_{B_{\delta}(x_n)}(|G(x,y)|+|G(x,x_n)|)dy
\nonumber\\
&\overset{G(x,\, \cdot\,)\in L^1}{\leq}&
\frac{C p_n^{\frac{13}{2}}}{\|u_{p_n}\|_{L^{\infty}(\Omega)}^{\frac12 (p_n-1)}}\underset{n\to+\infty}{\longrightarrow} 0,
\end{eqnarray}
where the last convergence is due to \eqref{sqrte}.
At last, recalling that $x\not\in B_{2\delta}(x_\infty)$:
\begin{equation}\label{set3}
\begin{aligned}
p_n^2\lambda_{i,n}&\int_{B_{\tfrac{1}{p_n^2}}(x_n)}|G(x,y)-G(x,x_n)| u_{p_n}^{{p_n}-1}(y)|v_{i,n}(y)| dy\\
&\leq p_n^2C \sup_{\xi\in B_{\tfrac{1}{p_n^2}}(x_n)}|\nabla G(x,\xi)|\int_{B_{\tfrac{1}{p_n^2}}(x_n)}|y-x_n| u_{p_n}^{{p_n}-1}(y) dy\\
&\leq  C\sup_{\xi\in B_{\tfrac{1}{p_n^2}}(x_n)}|\nabla G(x,\xi)|\int_\Omega u_{p_n}^{{p_n}-1}(y) dy\overset{\eqref{bound_int_p-1}}{\leq} \frac{C}{{p_n}}\sup_{\xi\in B_{\tfrac{1}{p_n^2}}(x_n)}|\nabla G(x,\xi)|\underset{n\to+\infty}{\longrightarrow 0}.
\end{aligned}
\end{equation}
Combining \eqref{set1}, \eqref{set2} and \eqref{set3} we get that $I_{i,n}(x)=o(1)$. It is not difficult to see that the convergence is $C^0_{loc}(\overline\Omega\setminus\{x_{\infty}\})$, the uniform convergence of the derivatives $p_n\frac{\partial v_{i,n}}{\partial x_j}$  may be done in a similar way, so we omit it.\\
\\\\\\
{\emph{Step 2. We derive \eqref{riorganizzazione2}.}}\\	
By the integral identity \eqref{integralidentity} with $y=x_n$, we have
\begin{equation}\label{ago2}
p_n\int_{\partial\Omega}(x-x_n)\cdot\nabla u_{p_n}\frac{\partial v_{i,n}}{\partial\nu}d\sigma_x=(1-\lambda_{i,n})p_n^2\int_\Omega  u_{p_n}^{{p_n}-1}v_{i,n}\left((x-x_n)\cdot\nabla u_{p_n}+\frac{2}{{p_n}-1}u_{p_n}\right)dx.
\end{equation}
By \eqref{convptimax}, Proposition \ref{prop:assurdob} and \eqref{convergenzapup} we can determine the  rate of decay of the l.h.s. of \eqref{ago2}:
\begin{eqnarray}\label{ago3}
p_n\int_{\partial\Omega}(x-x_n)\cdot\nabla u_{p_n}\frac{\partial v_{i,n}}{\partial\nu}d\sigma_x\!\!\!&=&\!\!\!-\frac{64\pi^2\sqrt{e}b^i}{p_n}(1+o(1))\int_{\partial \Omega}(x-x_\infty)\cdot\nu(x)\left(\frac{\partial G}{\partial\nu}(x,x_\infty)\right)^2 d\sigma_x\nonumber\\
&\overset{\eqref{Green1}}{=}&\!\!\!-\frac{32\pi\sqrt{e}b^i}{p_n}(1+o(1))
\end{eqnarray}
On the other hand recalling the definition of the rescaled functions $w_n:=w_{1,n}$ (see \eqref{defRiscalataMax}) and $\widetilde u_{p_n}$ (see \eqref{utilde})  we have
\begin{equation}
\label{ago4}
\begin{aligned}
&p_n^2\int_\Omega  u_{p_n}^{{p_n}-1}(x)v_{i,n}(x)\left((x-x_n)\cdot\nabla u_{p_n}+\frac{2}{{p_n}-1}u_{p_n}\right)dx=\\
&\overset{(\star)}{=}p_n^2\int_{B_r(x_n)} u_{p_n}^{{p_n}-1}(x)v_{i,n}(x)\left((x-x_n)\cdot\nabla u_{p_n}+\frac{2}{{p_n}-1}u_{p_n}\right)dx+o(1)\\
&=\int_{B_{\frac{r}{\varepsilon_n}}(0)}\left(
1+\frac{w_n(y)}{p_n}
\right)^{{p_n}-1}\!\!\! p_n\widetilde  v_{i,n}(y)\left(y\cdot\nabla \widetilde u_{p_n}(y)+\frac{2}{{p_n}-1}\widetilde u_{p_n}(y)\right)dy+o(1),
\end{aligned}
\end{equation}
where $r$ is as in the statement of Theorem \ref{teoAsymptotic} and in $(\star)$ we have used that $B_{\frac{r}{2}}(x_\infty)\subseteq B_{r}(x_n)$ for $n$ large by \eqref{convptimax}, so
that
\begin{eqnarray*}
	&&	\!\!\!\!\!\!\!\!\!\!\!\!\!	\!\!\!\!\!\!\!\!\!\!\!\!\!
		\!\!\!\!\!\!\!\!\!\!\!\!\!\!\!\!\!\!\!\!\!\!\!\!\!\!\!
		\left|p_n^2\int_{\Omega\setminus B_r(x_n)} 
	u_{p_n}^{{p_n}-1}(x)v_{i,n}(x)\left((x-x_n)\cdot\nabla u_{p_n}+\frac{2}{{p_n}-1}u_{p_n}\right)dx\right|\leq
	\\
	&\leq& p_n^2\int_{\Omega\setminus B_{\frac{r}{2}}(x_\infty)} 
	u_{p_n}^{{p_n}-1}(x)\left|(x-x_n)\cdot\nabla u_{p_n}+\frac{2}{{p_n}-1}u_{p_n}\right|dx
	\\
	&\overset{\eqref{P4}+\eqref{boundLoo}}{\leq}& |\Omega| p_n\|u_{p_n}\|^{p_n-1}_{L^\infty(\Omega\setminus B_{\tfrac r2}(x_\infty))}C\overset{\eqref{utile} }{=}o(1).
	\end{eqnarray*}
Arguing as in Remark \ref{rem:lebesgue}, we can pass to the limit into \eqref{ago4}. Indeed, setting now $g_n(y):=p_n \widetilde v_{i,n} (y)\left( y\cdot\nabla \widetilde u_{p_n}(y)+\frac{2}{{p_n}-1}\widetilde u_{p_n}(y)\right)$,
by \eqref{convutilde}, \eqref{convnablautilde} and \eqref{tesiLemmalimab} one  has $g(y):= \sqrt{e}(y\cdot \nabla U(y)+2)\left(\sum_{j=1}^2\frac{a^i_j y_j}{8+|y|^2}+ b^i\frac{8-|y|^2}{8+|y|^2}\right)$   
and by \eqref{P4tilde} and \eqref{boundLootilde}, one can take $h$ to be a constant function. Hence we get
\begin{equation}\label{ago5}
\begin{aligned}
\int_{B_{\frac r{\varepsilon_n}}(0)}&\left(
1+\frac{w_n(y)}{p_n}
\right)^{{p_n}-1}\!\!\!p_n\widetilde v_{i,n}\,\left(y\cdot\nabla \widetilde u_{p_n}(y)+\frac{2}{{p_n}-1}\widetilde u_{p_n}(y)\right)dy
\\
&= \sqrt{e}\int_{\R^2}(y\cdot \nabla U(y)+2)\left(\sum_{j=1}^2\frac{a^i_j y_j}{8+|y|^2}+ b^i\frac{8-|y|^2}{8+|y|^2}\right)dy +o(1)
\\
&\overset{\eqref{definizioneU}}{=}128\sqrt e \int_{\R^2}\frac{8-|y|^2}{(8+|y|^2)^3}\left(\sum_{j=1}^2\frac{a^i_j y_j}{8+|y|^2}+b^i\frac{8-|y|^2}{8+|y|^2}\right) dy+o(1)\\
&=\frac{16}{3}\pi\sqrt e b^i +o(1).
\end{aligned}
\end{equation} 
\\
Inserting \eqref{ago3}, \eqref{ago4} and \eqref{ago5} into \eqref{ago2} we derive that
\[
-\frac{32\pi\sqrt{e}b^i}{p_n}(1+o(1))=(1-\lambda_{i,n})\frac{16}{3}\pi \sqrt{e}b^i (1+o(1)).
\]
Since by assumption $b^i \neq0$, we find \eqref{riorganizzazione2}.
\end{proof}

\

\

\section{Proof of \eqref{ele2} and  \eqref{ele1}   of Theorem \ref{teoAutofunzioniAutovalori} }
 \label{section:ele12}
This section is devoted to the proof of \eqref{ele2} and  \eqref{ele1}   of Theorem \ref{teoAutofunzioniAutovalori}.\\

First we show that both the eigenvalues $\lambda_{2,n}$ and $\lambda_{3,n}$ of \eqref{eigenvalueProblem} converge to $1$, as $n\to+\infty$, and we also obtain a first  estimate on their asymptotic behavior (Proposition \ref{prop:1stimaautov} below). Then we prove a convergence result for the corresponding rescaled eigenfunctions $\widetilde v_{2,n}$ and  $\widetilde v_{3,n}$ (Proposition \ref{thm:eigenfunctions}). Finally at the end of the section we prove  \eqref{ele2}  and \eqref{ele1}  of Theorem \ref{teoAutofunzioniAutovalori}.

\

\begin{proposition}\label{prop:1stimaautov}
	We have 
	\begin{eqnarray}
	&\lambda_{i,n}&\leq\;\,1+C\varepsilon_n^2\label{1stimaautov}\\
	&\lambda_{i,n}&\longrightarrow\;\, 1 \qquad\qquad\qquad\qquad\quad\mbox{as $n\to+\infty$}\label{autova1}
	\end{eqnarray}
	for $i=2,3$.
\end{proposition}
\begin{proof}
By the variational characterization of the eigenvalues we have
	\begin{equation}\label{caratt}
	\lambda_{i,n}= \inf_{\substack{
			W\subset H^1_{0}(\Omega)\\ dim W=i}}   \max_{\substack{v\in W\\v\neq 0} }\  \frac{\int_{\Omega}|\nabla v|^2dx}{p_n\int_{\Omega}u_{p_n}^{p_n-1}v^2dx}
	\end{equation}
	Observe that the functions
	$\frac{\partial u_p}{\partial x_1}$, $\frac{\partial u_p}{\partial x_2}$
	solve the equation in \eqref{eigenvalueProblem} when $\lambda=1$, but not the boundary conditions, hence they are almost in the $Ker(L_p)$. We cut-off  and use them as test functions in order to estimate the eigenvalues $\lambda_{i,n}$.	Let  	$r>0$ be as in Theorem \ref{teoAsymptotic} and let $\widetilde\phi\in C^{\infty}_0(B_{r}(0))$ be such that $\widetilde\phi\equiv 1$ in $B_{\frac r2}(0)$, $0\leq\widetilde\phi\leq 1$ in 
	$B_{r}(0)$. Let us define the  functions of $H^1_0(\Omega)$
	\[
	\psi_{i,n}:= \frac{\partial u_{p_n}}{\partial x_i}\phi_n  \quad  i=1,2
	\]
	where 
$\phi_n(x):=\widetilde\phi(x-x_n)$ and  $x_n$ is as in \eqref{ptimaxloc1}, and let us denote by
	\[W_i:=span\{u_{p_n},\ \psi_{j,n},\ j=1,\ldots, i-1\},\quad i=2,3.\]
It is easy to see (similarly as in \cite[Lemma 3.1]{GrossiPacellaMathZ}) that the functions
	$u_{p_n}$, $\psi_{1,n}$, $\psi_{2,n}$ are linearly independent for $n$ sufficiently large, 
	then $dim W_i=i$. As a consequence, by \eqref{caratt}, it follows that 
	\begin{equation}\label{lambdaiMinoreMax}\lambda_{i,n}\leq \max_{\substack{v\in W_i\\v\neq 0} } \frac{\int_{\Omega}|\nabla v|^2dx}{p_n\int_{\Omega}u_{p_n}^{p_n-1}v^2dx}. 
	\end{equation}
	Let us evaluate separately $\int_{\Omega}|\nabla v|^2dx$ and $p_n\int_{\Omega}u_{p_n}^{p_n-1}v^2dx$ for a function $v\in W_i$, i.e. for
	\[
	v=a_0 u_{p_n} + \sum_{j=1}^{i-1}a_j\psi_{j,n}= a_0 u_{p_n} + \phi_n z_n
	\]
	where
	\begin{equation}\label{defzp}
	z_n:= \sum_{j=1}^{i-1}a_j\frac{\partial u_{p_n}}{\partial x_j}
	\end{equation}
	and $a_0, a_1, \ldots a_{i-1}\in\R$.
	We have
	\begin{eqnarray}\label{denominatore}
	p_n\int_{\Omega}u_{p_n}^{p_n-1}v^2dx&=&p_na_0^2\int_{\Omega}u_{p_n}^{p_n+1}dx+ 2p_na_0\int_{\Omega}u_{p_n}^{p_n}\phi_n z_n\,dx+ p_n \int_{\Omega}u_{p_n}^{p_n-1}\phi_n^2z_n^2\,dx
	\nonumber\\
	&=&p_na_0^2\int_{\Omega}u_{p_n}^{p_n+1}dx+ 2p_na_oA_{n,i} +C_{n,i} 
	\end{eqnarray}
	Moreover
	\begin{eqnarray}\label{numeratore}
	\int_{\Omega}|\nabla v|^2\,dx &=&  \int_{\Omega}|\nabla a_0 u_{p_n} + \phi_n z_n|^2\,dx
	\nonumber
	\\
	& =  & 
	a_0^2\int_{\Omega}|\nabla u_{p_n}|^2dx + 2 a_0\int_{\Omega}\nabla  u_{p_n} \nabla (\phi_n z_n)dx + \int_{\Omega}|\nabla (\phi_n z_n)|^2dx\nonumber
	\\
	&\overset{\eqref{problem}}{ =}  & 
	a_0^2\int_{\Omega} u_{p_n}^{p_n+1}dx + 2 a_0\underbrace{\int_{\Omega}\nabla  u_{p_n} \nabla (\phi_n z_n) dx}_{=:I_n} + \underbrace{\int_{\Omega}|\nabla (\phi_n z_n)|^2dx}_{=:II_n}.
	\end{eqnarray} 
	Multiplying \eqref{problem} by $\phi_n z_n$ and integrating we have
	\begin{eqnarray}\label{altroPezzo}
	I_n=\int_{\Omega}\nabla  u_{p_n} \nabla (\phi_n z_n)dx =\int_{\Omega}u_{p_n}^{p_n}\phi_n z_ndx.
	\end{eqnarray}
Moreover $z_n$ satisfies the equation 
	\[-\Delta z_n=p_nu_{p_n}^{p_n-1}z_n \ \mbox{ in }\Omega, \]
	multiplying it by $\phi_n^2z_n$ and integrating we have
	\begin{equation}\label{beto}
	\int_{\Omega}\phi_n^2 |\nabla z_n|^2 dx+ 2\int_{\Omega}\phi_n z_n \nabla\phi_n\cdot\nabla z_ndx = p_n\int_{\Omega} u_{p_n}^{p_n-1}\phi_n^2z_n^2dx
	\end{equation}
	and then
	\begin{eqnarray}\label{unPezzo}
	II_n&=&\int_{\Omega}|\nabla (\phi_n z_n)|^2 dx= 
	\int_{\Omega}|\nabla \phi_n|^2 z_n^2dx+\int_{\Omega}\phi_n^2|\nabla z_n|^2dx+ 2\int_{\Omega}
	\phi_n z_n \nabla\phi_n\cdot\nabla z_ndx
	\nonumber
	\\
	&\overset{\eqref{beto}}{=} & \int_{\Omega}|\nabla \phi_n|^2 z_n^2dx+ p_n\int_{\Omega} u_{p_n}^{p_n-1}\phi_n^2z_n^2dx.
	\end{eqnarray}
	Substituting \eqref{unPezzo} and \eqref{altroPezzo} into \eqref{numeratore} we then have
	\begin{equation}\label{numeratore2}
	\int_{\Omega}|\nabla v|^2dx=a_0^2\int_{\Omega}u_{p_n}^{p_n+1} dx+ 2 a_0\int_{\Omega}  u_{p_n}^{p_n} \phi_n z_n dx + \int_{\Omega}|\nabla \phi_n|^2 z_n^2dx+ p_n\int_{\Omega} u_{p_n}^{p_n-1}\phi_n^2z_n^2dx.
	\end{equation}
	By \eqref{numeratore2} and \eqref{denominatore} we get
	\begin{eqnarray} \label{primoConto}
	&&\max_{\substack{v\in W_i\\v\neq 0} } \frac{\int_{\Omega}|\nabla v|^2dx}{p_n\int_{\Omega}u_{p_n}^{p_n-1}v^2dx}=
	\nonumber
	\\
	&=&
	\max_{\substack{a_0,\ldots, a_{i-1}\in\mathbb R} } 1+\frac{(1-p_n)a_0^2\int_{\Omega}u_{p_n}^{p_n+1}dx + 2(1-p_n) a_0 A_{n,i} + B_{n,i}}{p_na_0^2\int_{\Omega}u_{p_n}^{p_n+1}dx+ 2p_na_0 A_{n,i}+ C_{n,i}},
	\end{eqnarray}
	where
\begin{eqnarray}&&\label{defAnn}A_{n,i}:=\int_{\Omega}  u_{p_n}^{p_n} \phi_n \sum_{j=1}^{i-1}a_j\frac{\partial u_{p_n}}{\partial x_j} dx, 
				\\
&&\label{defBnn} B_{n,i}:=\int_{\Omega}|\nabla \phi_n|^2\left(\sum_{j=1}^{i-1}a_j\frac{\partial u_{p_n}}{\partial x_j}\right)^2dx, \\ && \label{defCnn} C_{n,i}:=p_n\int_{\Omega}u_{p_n}^{p_n-1}\phi_n^2\left(\sum_{j=1}^{i-1}a_j\frac{\partial u_{p_n}}{\partial x_j}\right)^2dx.\end{eqnarray} 
\\\\
\emph{Step 1. We estimate the terms $A_{n,i}$, $B_{n,i}$ and $C_{n,i}$.}
	\\\\
	Integrating by part,  since $\nabla\phi_n$ vanishes in $B_{\frac r2}(x_n)$ 
	\begin{eqnarray*}
		A_{n,i}&=&
		\frac{1}{p_n+1}\int_{\Omega}\phi_n \left(\sum_{j=1}^{i-1}\frac{\partial (u_{p_n}^{p_n+1})}{\partial x_j}\right)dx
		=
		-\frac{1}{p_n+1}\int_{\Omega}\left(\sum_{j=1}^{i-1}\frac{\partial \phi_n }{\partial x_j}\right)u_{p_n}^{p_n+1}dx\nonumber\\
		&=&
		-\frac{1}{p_n+1}\int_{\Omega\cap \{|x-x_n|\geq\frac r2\}}\left(\sum_{j=1}^{i-1}\frac{\partial \phi_n }{\partial x_j}\right)u_{p_n}^{p_n+1}dx,
	\end{eqnarray*}
	so, since $x_n\rightarrow x_{\infty}$ and using the convergence in \eqref{convergenzapup}
	\begin{eqnarray}
	\label{stimaTermineA}
	|A_{n,i}|&\leq & 
	\frac{1}{p_n+1}\int_{\Omega\cap \{|x-x_n|\geq\frac r2\}}\left(\sum_{j=1}^{i-1}\left|\frac{\partial \phi_n }{\partial x_j}\right|\right)u_{p_n}^{p_n+1}dx
	\nonumber
	\\
	&\overset{x_n\rightarrow x_{\infty}}{\leq} & \frac{C}{p_n(p_n+1)}\, p_n\int_{\Omega\cap \{ |x-x_{\infty}|\geq\frac r4\}}u_{p_n}^{p_n+1}dx
	\nonumber 
	\\
	&\overset{\eqref{P3}}{\leq}& \frac{C_{r}}{p_n^3(p_n+1)}\int_{\Omega\cap \{ |x-x_{\infty}|\geq\frac r4\}}(p_nu_{p_n})^2dx\overset{\eqref{convergenzapup}}{=}
	O\left(\frac{1}{p_n^4}\right).
	\end{eqnarray}
	Similarly, using again that $x_n\rightarrow x_{\infty}$  and  the convergence in \eqref{convergenzapup} 
	\begin{eqnarray}\label{stimaTermineB}
	B_{n,i}&=&\frac{1}{p_n^2}
	\int_{\Omega\cap \{|x-x_n|\geq\frac r2\}}|\nabla \widetilde\phi(x-x_n)|^2 \left(\sum_{j=1}^{i-1}a_j p_n\frac{\partial u_{p_n}}{\partial x_j}\right)^2 dx
	\nonumber\\
	&\overset{\eqref{convergenzapup}}{=}& \frac{1}{p_n^2}(\widetilde C+o(1))
	\end{eqnarray}
	where
	\[\widetilde C:=64\pi^2 e\int_{\Omega}|\nabla \widetilde\phi(x-x_{\infty})|^2 \left(\sum_{j=1}^{i-1}a_j\frac{\partial G(x,x_{\infty})}{\partial x_j}\right)^2dx>0.\]
	Finally, by changing variable and recalling the definition of $w_n$
	\begin{eqnarray}\label{stimaTermineC}
	C_{n,i}&=&p_n\int_{\Omega}u_{p_n}^{p_n-1}\phi_n^2\left(\sum_{j=1}^{i-1}a_j\frac{\partial u_{p_n}}{\partial x_j}\right)^2dx
	\nonumber
	\\
	&=&
	\frac{\|u_{p_n}\|_{L^{\infty}(\Omega)}^{p_n-1}}{p_n}\int_{B_{\frac{r}{\varepsilon_n}}(0)}
	\left(1+\frac{w_n}{p_n} \right)^{p_n-1}\|u_{p_n}\|_{L^{\infty}(\Omega)}^{2}\widetilde\phi (\varepsilon_n y)^2\left(\sum_{j=1}^{i-1}a_j\frac{\partial w_n}{\partial y_j}\right)^2dy
	\nonumber
	\\
	&\overset{(\star)}{=}&\widehat C\, \frac{\|u_{p_n}\|_{L^{\infty}(\Omega)}^{p_n-1}}{p_n}(1+o_n(1))
	\end{eqnarray}
	with
	\[\widehat C:= e \int_{\mathbb R^2}e^U\left(\sum_{j=1}^{i-1}a_j\frac{\partial U}{\partial y_j}\right)^2>0,\] 
	where the passage to the limit in $(\star)$ can be justified arguing as in Remark \ref{rem:lebesgue}. Indeed, setting  $g_n:= \|u_{p_n}\|_{L^{\infty}(\Omega)}^{2}\widetilde \phi (\varepsilon_n y)^2\left(\sum_{j=1}^{i-1}a_j\frac{\partial w_n}{\partial y_j}\right)^2$,
	by \eqref{sqrte} and \eqref{convRiscalateNeiMax} one has $g:= e \left(\sum_{j=1}^{i-1}a_j\frac{\partial U}{\partial y_j}\right)^2$. Moreover, recalling the definition of $\widetilde u_{p_n}$ (see \eqref{utilde}), one has
	$\frac{\partial w_n}{\partial y_j}=\|u_{p_n}\|_{L^{\infty}(\Omega)}^{-1}p_n\frac{\partial \widetilde u_{p_n}}{\partial y_j}$, hence	
	\[\left|\frac{\partial w_n(y)}{\partial y_j}\right|\overset{\eqref{P4tilde}}{\leq}\frac{C}{\|u_{p_n}\|_{L^{\infty}(\Omega)}}\frac{1}{|y|}\overset{\eqref{sqrte}}{\leq} C \frac{1}{|y|} \quad \mbox{ for }y\in{\Omega}_n,\]
	moreover, for $R>0$ fixed, by \eqref{convRiscalateNeiMax}
	\[\left\|\frac{\partial w_n}{\partial y_j}\right\|_{L^{\infty}(B_R(0))}\leq C,\]
	as a consequence  we can take 
	\[h(y):= C \left\{
	\begin{array}{lr}
	1 &  \mbox{  for }|y|\leq R
	\\
	\frac{1}{|y|^2} & \mbox{ for }R<|y|\leq\frac{r}{\varepsilon_n},
	\end{array}
	\right.
	\] for a suitable constant $C>0$, and apply  Remark \ref{rem:lebesgue}.  
	\\\\\\
	\emph{Step 2. We prove that there exists $C_0>0$ such that
		\begin{equation}\label{stimaStep2}
		\max_{\substack{v\in W_i\\v\neq 0} }\  \frac{\int_{\Omega}|\nabla v|^2 dx}{p_n\int_{\Omega}u_{p_n}^{p_n-1}v^2dx}\geq 1+ C_0\varepsilon_n^2.
		\end{equation}}
	\\
	Choosing $(a_0,\ldots, a_{i-1})=(0,1,\ldots,1)$ we deduce by \eqref{primoConto}
	\[
	\max_{\substack{v\in W_i\\v\neq 0} }\  \frac{\int_{\Omega}|\nabla v|^2dx}{p_n\int_{\Omega}u_{p_n}^{p_n-1}v^2dx}\geq 1+ \frac{B_{n,i}}{C_{n,i}},
	\]
	where $B_{n,i}$ and $C_{n,i}$ are defined in \eqref{defBnn} and \eqref{defCnn} respectively.
	Hence using  the estimates \eqref{stimaTermineB} and \eqref{stimaTermineC}  we get
	\[
	\max_{\substack{v\in W_i\\v\neq 0} }\  \frac{\int_{\Omega}|\nabla v|^2dx}{p_n\int_{\Omega}u_{p_n}^{p_n-1}v^2dx}
	\geq 
	1+ \frac{\widetilde C(1+o_n(1))}{\widehat C p_n\|u_{p_n}\|_{L^{\infty}(\Omega)}^{p_n-1}(1+o_n(1))}\geq 1+\varepsilon_n^2C_0,
	\]
	where $C_0:=\frac{1}{2}\frac{\widetilde C}{\widehat C}$. 
	\\
	\\
	\\
	\emph{Step 3. Proof of \eqref{1stimaautov}.}
	\\
	Let us denote by $a_{0,n}, a_{1,n},\ldots, a_{i-1,n}$ the point of $\mathbb R^i$ where the quotient in \eqref{primoConto} achieves its maximum. By \eqref{primoConto}, the limit of the energy in \eqref{convergenzaenergia} and the estimates in \eqref{stimaTermineA}, \eqref{stimaTermineB} and \eqref{stimaTermineC} we derive
	\begin{eqnarray}\label{confrontoOrdini}
	\max_{\substack{v\in W_i\\v\neq 0} }\  \frac{\int_{\Omega}|\nabla v|^2dx}{p_n\int_{\Omega}u_{p_n}^{p_n-1}v^2dx}
	&=&
	1+\frac{-a_{0,n}^2(8\pi e+ o_n(1))-2a_{0,n}O\left(\frac{1}{p_n^3}\right)+\frac{\widetilde C}{p_n^2}(1+o(1))}{a_{0,n}^2(8\pi e+ o_n(1))+2a_{0,n}O\left(\frac{1}{p_n^3} \right)+\widehat C\frac{\|u_{p_n}\|_{L^{\infty}(\Omega)}^{p_n-1}}{p_n}(1+o_n(1))}
	\nonumber
	\\
	&=& 1+\varepsilon_n^2\, 
	\frac{-a_{0,n}^2p_n^2(8\pi e+ o_n(1))-2a_{0,n}p_nO\left(\frac{1}{p_n^2}\right)+\widetilde C(1+o(1))}{
		\varepsilon_n^2 a_{0,n}^2 p_n^2(8\pi e+ o_n(1))+2a_{0,n}p_n\varepsilon_n^2O\left(\frac{1}{p_n^2} \right)+ \widehat C(1+o_n(1))}.
	\end{eqnarray}
	As a consequence 
	\begin{equation}\label{bounda0}|a_{0,n} p_n|\leq C \end{equation}
	otherwise from \eqref{confrontoOrdini} one gets a contradiction with \eqref{stimaStep2}.
	From \eqref{lambdaiMinoreMax}, \eqref{confrontoOrdini} with \eqref{bounda0} we get the conclusion
	\[\lambda_{i,n}\leq1+\varepsilon_n^2 C.\]	
	\emph{Step 4. Proof of \eqref{autova1}.}
	\\
	By \eqref{1stimaautov} it is enough to prove
	\[\lambda_{2,n}\rightarrow 1 \mbox{ as }n\rightarrow +\infty.\]
	Observe that \[\frac{1}{p_n}=\lambda_{1,n}\leq\lambda_{2,n}\overset{\eqref{1stimaautov}}{\leq} 1+\varepsilon_nC\]
	Hence, up to a subsequence, $\lambda_{2,n}\rightarrow\Lambda\in [0,1]$. Assume by contradiction that 
	\begin{equation}\label{LambdaMin1}
	\Lambda<1.
	\end{equation}
	Let us consider an eigenfunction $v_{2,n}$ related to $\lambda_{2,n}$, which solves \eqref{eigenvalueProblem}, and let $\widetilde v_{2,n}$ be its rescaling defined as in \eqref{rescaledeigenfunct} and which solves \eqref{eqrescaledeigenfunct}.\\
	It is easy to show that $\nabla \widetilde v_{2,n}$ is uniformly bounded in $L^2(\mathbb R^2)$, indeed:
	\begin{eqnarray*}
		\int_{\mathbb R^2}|\nabla \widetilde v_{2,n}|^2dy
		&=&
		\int_{\Omega_n} |\nabla \widetilde v_{2,n}|^2dy
		\overset{\eqref{eqrescaledeigenfunct}}{=}
		\lambda_{2,n}
		\int_{\Omega_n}
		\left( 1+\frac{w_n}{p_n}\right)^{p_n-1}\!\!\!\!\!\widetilde v_{2,n}^2dy
		\leq  \int_{\Omega_n}
		\left( 1+\frac{w_n}{p_n}\right)^{p_n-1}\!\!\!\!\!\!\!\!\!dy
		\\
		&=&p_n\int_{\Omega} u_{p_n}^{p_n-1}dx
		\overset{\eqref{bound_int_p-1}}{\leq} C.
	\end{eqnarray*}
	So by the standard elliptic regularity theory $\widetilde v_{2,n}\rightarrow \widetilde v$ in $C^1_{loc}(\mathbb R^2)$ where $\widetilde v\neq 0$ by Lemma \ref{lemma:varphineq0} and it is a solution of the limit eigenvalue problem
	\begin{equation}\label{linearLiouville}
	\left\{
	\begin{array}{lr}
	-\Delta \widetilde v =\Lambda e^U\widetilde v\quad\mbox{ in }\mathbb R^2
	\\
	\|\widetilde v\|_{L^{\infty}(\R^2)}\leq 1.
	\end{array}
	\right.
	\end{equation}
	Taking the stereographic projection on the $S^1$ sphere, problem \eqref{linearLiouville} is  reduced to the eigenvalue problem for the Laplace-Beltrami operator $-\Delta_{S^1}$ with the same  eigenvalue $\Lambda$, by \eqref{LambdaMin1} it then follows,  that $\Lambda=0$ and so \begin{equation}\label{limcost}\widetilde v=const\neq 0.\end{equation}
	Observe also that
	\begin{eqnarray}\label{opicco}
	\int_{\Omega\setminus B_r(x_n)}u_{p_n}^{p_n}|v_{2,n}|dx &\leq&  |\Omega|
	\|u_{p_n}\|^{p_n-1}_{L^\infty(\Omega\setminus B_r(x_n))}\|u_{p_n}\|_{L^{\infty}(\Omega)}
	\nonumber\\
	&\overset{(\star)}{\leq}& 2 |\Omega|
	\|u_{p_n}\|^{p_n-1}_{L^\infty(\Omega\setminus B_{\frac{r}{2}}(x_\infty))}\|u_{p_n}\|_{L^{\infty}(\Omega)}
	\nonumber\\
	&\overset{\eqref{utile}+\eqref{boundLoo}}{=}& o(1),
	\end{eqnarray}
	where in $(\star)$ we have used that $B_{\frac{r}{2}}(x_{\infty})\subseteq B_r(x_n)$ for $n$ large, which is a consequence of \eqref{convptimax}. 
	By \eqref{limcost}, using \eqref{opicco} and the dominated convergence theorem (similarly as in Remark \ref{rem:lebesgue}, being $\|\widetilde v_{2,n}\|_{L^{\infty}(\Omega_n)}=1$), one has: 
	\begin{eqnarray*}
		p_n\int_{\Omega}\nabla \frac{u_{p_n}}{\|u_{p_n}\|_{L^{\infty}(\Omega)}}\nabla v_{2,n}dx &\overset{\eqref{problem}}{=}&\frac{p_n}{\|u_{p_n}\|_{L^{\infty}(\Omega)}}\int_{\Omega}u_{p_n}^{p_n}v_{2,n}dx
		\\
		&\overset{\eqref{opicco}}{=}&\frac{p_n}{\|u_{p_n}\|_{L^{\infty}(\Omega)}}\int_{B_r(x_n)}u_{p_n}^{p_n}v_{2,n}dx + o(1)
		\\
		&=&\int_{B_\frac{r}{\varepsilon_n}(0)}\left( 1+\frac{w_n}{p_n}\right)^{p_n} \widetilde v_{2,n}dy + o(1)
		\\
		&\overset{\eqref{convRiscalateNeiMax}}{=}& \widetilde v \int_{\R^2}e^U dy +o(1) \neq 0,
	\end{eqnarray*}
	but this is in contradiction with the orthogonality in $H^1_0(\Omega)$ of the eigenfunctions $v_{1,n}=\frac{u_{p_n}}{\|u_{p_n}\|_{L^{\infty}(\Omega)}}$ and $v_{2,n}$ (see \eqref{autofunzOrtogonali}).
\end{proof}

\

\begin{proposition} \label{thm:eigenfunctions}
	Let $\widetilde{v}_{i,n}$ be the rescaled eigenfunction defined in \eqref{rescaledeigenfunct}, we have:
	\begin{equation}\label{vtildeasenzab}
	\widetilde{v}_{i,n}(y)\longrightarrow\sum_{j=1}^2\frac{a^i_j y_j}{8+|y|^2}\qquad \mbox{as $n\to+\infty$\quad in $C^1_{loc}(\R^2)$}
	\end{equation}
	for $i=2,3$, for some vectors $a^i=(a_1^i,a_2^i)\neq0$ in $\R^2$,  $a^2$ and $a^3$  orthogonal in $\R^2$.	
\end{proposition}
\begin{proof}
	Applying Proposition \ref{prop:1stimaautov} and Lemma \ref{lemma:limab} we get
the existence of $a^i_1$, $a^i_2$, $b^i\in\R$, $	(a^i_1,a^i_2,b^i)\neq(0,0,0)$ such that
	\[
	\widetilde{v}_{i,n}(y)\longrightarrow\sum_{j=1}^2\frac{a^i_j y_j}{8+|y|^2}+ b^i\frac{8-|y|^2}{8+|y|^2}\quad \mbox{ as } n\to+\infty \mbox{ in }C^1_{loc}(\R^2).
	\]
Assume by contradiction that $b^i\neq 0$, then by \eqref{riorganizzazione2} in Proposition \ref{prop:assurdob} we have that for $n$ sufficiently large
	\[
	\lambda_{i,n}\geq 1+\frac{3}{p_n},
	\]
	but this is in contradiction with the estimate \eqref{1stimaautov} in Proposition \ref{prop:1stimaautov}, hence necessarily $b^i=0$ and \eqref{vtildeasenzab} holds.\\
	Next we show the orthogonality of the vectors $a^2$ and $a^3$.\\
	By assumption $\int_\Omega \nabla v_{2,n} \nabla v_{3,n}dx=0$. Using equation \eqref{eigenvalueProblem} we get
	\[
	\int_\Omega p_n u_{p_n}^{p_n-1} v_{2,n}v_{3,n} dx=0,
	\]
	so by \eqref{utile}
	\[
	\int_{B_r(x_n)} p_n u_{p_n}^{p_n-1} v_{2,n}v_{3,n} dx+o(1)=0,
	\]
	where $r$ is as in the statement of Theorem \ref{teoAsymptotic}. Rescaling and recalling the definition of $w_n:=w_{1,n}$ (see \eqref{defRiscalataMax}) and of $\widetilde v_{i,n}$ (see \eqref{rescaledeigenfunct}) we obtain
	\[
	\int_{B_{\frac{r}{\varepsilon_n}}}\left(1+\frac{w_n}{p_n}\right)^{{p_n}-1}\widetilde v_{2,n}\widetilde v_{3,n} dz+o(1)=0.
	\]
	Using \eqref{convRiscalateNeiMax} and the convergence in  \eqref{vtildeasenzab}, arguing as in 
	Remark \ref{rem:lebesgue} with 
	$g_n:=\widetilde v_{2,n}\widetilde v_{3,n}$, $g:=\frac{(a_1^2 z_1+a_2^2 z_2)(a_1^3 z_1+a_2^3 z_2)}{(8+|z|^2)^2}$  and $h=1$ (since $\|\widetilde v_{i,n}\|_{L^{\infty}(\Omega_n)}=1$), we can pass to the limit and get
	\[
	\int_{\R^2} e^{U(z)}\frac{(a_1^2 z_1+a_2^2 z_2)(a_1^3 z_1+a_2^3 z_2)}{(8+|z|^2)^2}dz=0,
	\]
	hence, by \eqref{definizioneU},
	\[
	\sum_{h,k=1}^2 a_h^2 a_k^3\int_{\R^2}\frac{z_h z_k}{(8+|z|^2)^4}dz=0,
	\]
	which implies
	\[
	\sum_{h=1}^2 a_h^2 a_h^3=0,
	\]
namely that the vectors $a^2$ and $a^3$ are orthogonal in $\R^2$.	
\end{proof}

\

\subsection{Proof of \eqref{ele2} in Theorem \ref{teoAutofunzioniAutovalori}}	

\begin{proof}
  Let us chose
	\begin{equation}
	\label{rhop}
	\tau_n=o(1)\qquad\mbox{such that}\qquad \frac{\varepsilon_n}{\tau_n^{7}}\to 0.
	\end{equation}
Observe that	
	\begin{equation}\label{rhonuo}\frac{r}{\varepsilon_n}\geq\frac{\tau_n}{\varepsilon_n}\rightarrow +\infty.\end{equation}
	\\
	\\
	{\emph{Step.1 We show that
	\begin{equation}\label{tesinaSTep1}
	v_{i,n}= E_n+F_n+ o(\varepsilon_n) \quad \mbox{  in }C^1_{loc}(\overline\Omega\setminus\{x_{\infty}\}),\end{equation}
where
\begin{equation}\label{E_n}
E_n(x):=\lambda_{i,n} G(x,x_n) \int_{B_{\tau_n}(x_n)}{p_n} u_{p_n}^{p_n-1}(y) v_{i,n}(y) dy
\end{equation}
and 
\begin{equation}\label{F_n}
F_n(x):=\lambda_{i,n}\sum_{j=1}^2 \frac{\partial G}{\partial y_j}(x,x_n)\int_{B_{\tau_n}(x_n)}p_n u_{p_n}^{p_n-1}(y) v_{i,n}(y)(y-x_n)_j dy.
\end{equation}}}

\

	For any $x\neq x_\infty$ there exists $\delta\in(0,r)$ (where $r$ is as in the statement of Theorem \ref{teoAsymptotic}) such that $x\not\in B_{2\delta}(x_\infty)\subset\Omega$.
	
	Using the Green's representation formula we have
	\begin{eqnarray}\label{vip1}
	v_{i,n}(x)&=&\lambda_{i,n}\int_\Omega G(x,y)p_n u_{p_n}^{p_n-1}(y) v_{i,n}(y) dy\nonumber\\
	&=&\lambda_{i,n}\int_{\Omega\setminus B_\delta(x_n)} G(x,y)p_n u_{p_n}^{p_n-1}(y) v_{i,n}(y) dy\nonumber\\
	&&+\lambda_{i,n}\int_{B_\delta(x_n)\setminus B_{\tau_n}(x_n)} G(x,y)p_n u_{p_n}^{p_n-1}(y)v_{i,n}(y) dy\nonumber\\
	&&+\lambda_{i,n}\int_{B_{\tau_n}(x_n)} G(x,y)p_n u_{p_n}^{p_n-1}(y) v_{i,n}(y) dy\nonumber\\
	&=&A_{n}(x)+B_{n}(x)+C_{n}(x).
	\end{eqnarray}
	Let us estimate the three terms separately.
	\begin{eqnarray}
	\label{Iip}
	|A_{n}(x)|&:= & |\lambda_{i,n}\int_{\Omega\setminus B_\delta(x_n)} G(x,y)p_n u_{p_n}^{p_n-1}(y) v_{i,n}(y) dy|
	\nonumber 
	\\
	&\leq &
	C\int_{\Omega\setminus B_\delta(x_n)} |G(x,y)|p_n u_{p_n}^{p_n-1}(y)  dy
	\nonumber
	\\
	&	\leq & p_n\|u_{p_n}\|^{p_n-1}_{L^\infty(\Omega\setminus B_{\frac \delta2}(x_\infty))}\int_\Omega |G(x,y)| dy
	\nonumber
	\\
	&\leq& p_n(\frac{1}{\sqrt{p_n}})^{p_n-1}C=o(\varepsilon_n),
	\end{eqnarray}
	where in the second inequality we have used that for $p$ sufficiently large $\|u_{p_n}\|^{p_n-1}_{L^\infty(\Omega\setminus B_{\frac \delta2}(x_\infty))}\leq \frac{1}{\sqrt{p_n}}$ by \eqref{pu_va_a_zeroTeo} and that $G(x,\,\cdot\,)\in L^1_y(\Omega)$.
	\\
	For the term $B_{n}$ we rescale, use the definition of $w_n:=w_{1,n}$ (see \eqref{defRiscalataMax}), the estimate \eqref{lebesgue} in Lemma \ref{lemma:lebesgue} and get (choosing $\gamma=\frac12$):
	\begin{eqnarray}\label{IIip}
	|B_{n}(x)|&:=& |\lambda_{i,n}\int_{B_\delta(x_n)\setminus B_{\tau_n}(x_n)} G(x,y)p_n u_{p_n}^{p_n-1}(y)v_{i,n}(y) dy|
	\nonumber
	\\
	&\leq&\int_{B_{\tfrac{\delta}{\varepsilon_n}}(0)\setminus B_{\tfrac{\tau_n}{\varepsilon_n}}(0)}|G(x,x_n+\varepsilon_n z)|\left(1+\frac{w_n(z)}{p_n}\right)^{p_n-1}dz\nonumber\\
	&\stackrel{\eqref{lebesgue}}{\leq}&\int_{B_{\tfrac{\delta}{\varepsilon_n}}(0)\setminus B_{\tfrac{\tau_n}{\varepsilon_n}}(0)}|G(x,x_n+\varepsilon_n z)|\frac{C}{1+|z|^{7/2}}dz\nonumber\\
	&\leq&\int_{B_{\tfrac{\delta}{\varepsilon_n}}(0)\setminus B_{\tfrac{\tau_n}{\varepsilon_n}}(0)}|G(x,x_n+\varepsilon_n z)|\frac{C}{|z|^{7/2}}dz\nonumber\\
	&\leq&C\frac{\varepsilon_n^{7/2}}{\tau_n^{7/2}}\int_{B_{\tfrac{\delta}{\varepsilon_n}}(0)\setminus B_{\tfrac{\tau_n}{\varepsilon_n}}(0)}|G(x,x_n+\varepsilon_n z)|dz\nonumber\\
	&\leq&C\frac{\varepsilon_n^{3/2}}{\tau_n^{7/2}}\int_{\Omega}|G(x,y)|dy\nonumber\\
	&\stackrel{G\in L^1_y}{\leq}& \varepsilon_n \sqrt{ \frac{\varepsilon_n}{\tau_n^7}} C\overset{\eqref{rhop}}
	{=}o(\varepsilon_n).
	\end{eqnarray}
	For any $y\in B_{\tau_n}(x_n)$ and $x\notin B_{2\delta}(x_\infty)$ the function $G$ is regular and we can expand it in Taylor series:
	\begin{equation}
	\label{agoGreen}
	G(x,y)=G(x,x_n)+\sum_{j=1}^2\frac{\partial G}{\partial y_j}(x,x_n)(y-x_n)_j+\frac12 \sum_{j,k=1}^2 \frac{\partial^2 G}{\partial y_j \partial y_k}(x,\eta_n)(y-x_n)_j(y-x_n)_k,
	\end{equation}
	where $\eta_n$ is a point on the line between $y$ and $x_n$, so $\eta_n\in B_{\tau_n}(x_n)$. As a consequence
	\begin{eqnarray}\label{agoGreenIsa}
	C_{n}(x)&:=& \lambda_{i,n}\int_{B_{\tau_n}(x_n)} G(x,y)p_n u_{p_n}^{p_n-1}(y) v_{i,n}(y) dy
	\nonumber
	\\
		&\overset{\eqref{agoGreen}}{=}& D_n(x)	+E_n(x) + F_n(x),
\end{eqnarray}
where $E_n$ and $F_n$ are defined in \eqref{E_n} and \eqref{F_n} respectively and 
\[D_n(x):=
\frac{\lambda_{i,n}}{2}\int_{B_{\tau_n}(x_n)} \sum_{j,k=1}^2 \frac{\partial^2 G}{\partial y_j \partial y_k}(x,\eta_n)(y-x_n)_j(y-x_n)_k {p_n} u_{p_n}^{p_n-1}(y) v_{i,n}(y) dy.\]
We now prove that $D_n=o(\varepsilon_n)$.  Notice that, since $x\notin B_{2\delta}(x_\infty)$ and $\eta_n\in B_{\tau_n}(x_n)\subset B_\delta(x_\infty)$, we have
\[
\left| \frac{\partial^2 G}{\partial y_j \partial y_k}(x,\eta_n)\right|\leq \sup_{y\in B_\delta(x_\infty),\, j,k=1,2}\left| \frac{\partial^2 G}{\partial y_j \partial y_k}(x,y)\right|= C,
\]
so we get
\begin{eqnarray}
\label{Rip}
|D_{n}(x)|&\leq& C\tau_n\int_{B_{\tau_n}(x_n)}{p_n} u_{p_n}^{{p_n}-1}(y)|y-x_n| dy
\nonumber
\\&=& C\tau_n\varepsilon_n\int_{B_{\frac{\tau_n}{\varepsilon_n}(0)}}\left(1+\frac{w_n(z)}{p_n}\right)^{{p_n}-1}|z| dz\nonumber\\
&\overset{(\star)}{=}& C\tau_n \varepsilon_n \left( \int_{\R^2} e^{U(z)}|z| dz= C\tau_n\varepsilon_n+o(1)\right)\overset{\eqref{rhop}}{=}o(\varepsilon_n),
\end{eqnarray}
where the convergence  in $(\star)$ is due to \eqref{convRiscalateNeiMax} and the passage to the limit  is possible by virtue of Remark \ref{rem:lebesgue}, observing that \eqref{rhonuo} holds.
\\
The proof of \eqref{tesinaSTep1} follows  substituting  \eqref{Rip},  \eqref{agoGreenIsa}, \eqref{IIip} and \eqref{Iip} into \eqref{vip1}. \\
Observe that we have proved the $C^0_{loc}(\overline\Omega\setminus\{x_{\infty}\})$ convergence, the uniform convergence of the derivatives $\frac{\partial v_{i,n}}{\partial x_j}$, $j=1,2$ may be done in a similar way, so we omit it.
\\
\\
\\
{\emph{Step 2. We show that
\begin{equation}\label{F_nopiccolo}
F_n
=\varepsilon_n2\pi\sum_{j=1}^2 a^i_j\frac{\partial G}{\partial x_j}(\,\cdot\,,x_\infty) + o(\varepsilon_n)\, \mbox{ as }n\rightarrow+\infty \mbox{ in }C^1_{loc}(\overline\Omega\setminus\{x_{\infty}\}).\end{equation}
}}
\begin{eqnarray*}
	\frac{F_n(x)}{\varepsilon_n}&=&\frac{\lambda_{i,n}}{\varepsilon_n}\sum_{j=1}^2 \frac{\partial G}{\partial y_j}(x,x_n)\int_{B_{\tau_n}(x_n)}{p_n} u_{p_n}^{p_n-1}(y) v_{i,n}(y) (y-x_n)_j dy
\\
&=&(1+o(1))\left(\sum_{j=1}^2 \frac{\partial G}{\partial y_j}(x,x_\infty)+o(1)\right)\int_{B_{\frac{\tau_n}{\varepsilon_n}}(0)}\left(1+\frac{w_n(z)}{p_n}\right)^{{p_n}-1}\widetilde{v}_{i,n} z_j dz
\nonumber
\\
&\overset{(\star)}{=}&
\sum_{j=1}^2 \frac{\partial G}{\partial y_j}(x,x_\infty)\int_{\R^2}e^U \frac{a_1^i z_1+a_2^i z_2}{8+|z|^2}z_jdz+o(1)
\nonumber\\
&\overset{(\star\star)}{=}&2\pi \sum_{j=1}^2 a_j^i \frac{\partial G}{\partial y_j}(x,x_\infty)+o(1)	
\end{eqnarray*}
where the convergence in $(\star)$ is due to \eqref{convRiscalateNeiMax}, \eqref{vtildeasenzab} and \eqref{rhonuo} and the passage to the limit is allowed by Remark \ref{rem:lebesgue}, taking
$g_n(z):=\widetilde v_{i,n}z_j$, $g(z):=\frac{a_1^i z_1+a_2^i z_2}{8+|z|^2}z_j$ (by  \eqref{vtildeasenzab}) and $h(z)=|z|$ (since $\|\widetilde v_{i,n}\|_{L^{\infty}(\Omega_n)}=1$).
While the equality in  $(\star\star)$ is a consequence of the  definition of $U$ (see \eqref{definizioneU}), computing explicitely the integral.
It is not difficult to see that the convergence is  $C^0_{loc}(\overline\Omega\setminus\{x_{\infty}\})$, moreover in a similar way one can prove the $C^1_{loc}(\overline\Omega\setminus\{x_\infty\})$ convergence.
\\
\\
\\
{\emph{Step 3. We prove that \begin{equation}\label{tesinaSTEP3}E_n=o(\varepsilon_n)\quad\mbox{  in }C^1_{loc}(\overline\Omega\setminus\{x_{\infty}\}).\end{equation}
}}
\\
\\
By a change of variable and using \eqref{convptimax} we get
\begin{eqnarray}\label{combi}
E_n(x)
&=&\lambda_{i,n} G(x,x_n) \int_{B_{\tau_n}(x_n)}{p_n} u_{p_n}^{p_n-1}(y) v_{i,n}(y) dy\nonumber\\
&=&
 (1+o(1))G(x,x_\infty)\int_{B_{\frac{\tau_n}{\varepsilon_n}}(0)}\left(1+\frac{w_n(z)}{p_n}\right)^{p_n-1}\widetilde v_{i,n}(z) dz
 \nonumber
 \\
 &\overset{(\star)}{=}&  (1+o(1))G(x,x_\infty)\left(\int_{\R^2}e^{U}\frac{a_1^i z_1+a_2^i z_2}{8+|z|^2} dz+o(1)\right)=o(1)
\end{eqnarray}
where the passage to the limit in $(\star)$ is due to
\eqref{convRiscalateNeiMax}, \eqref{vtildeasenzab} and \eqref{rhonuo} and follows  by Remark \ref{rem:lebesgue}, taking
$g_n(z):=\widetilde v_{i,n}$, $g(z):=\frac{a_1^i z_1+a_2^i z_2}{8+|z|^2}$  and $h(z)=1$.
Let us define
	\[
	\gamma_{i,n}=\int_{B_{\frac{\tau_n}{\varepsilon_n}}(0)}\left(1+\frac{w_n(z)}{p_n}\right)^{{p_n}-1}\widetilde{v}_{i,n}(z) dz,
	\]
We prove that 
\begin{equation}\label{gammaeps}\gamma_{i,n}=o(\varepsilon_n).
\end{equation} 
Finally \eqref{gammaeps}, combined with \eqref{combi},  implies that $E_n(x)=o(\varepsilon_n)$, moreover it is not difficult to see that this convergence is  $C^0_{loc}(\overline\Omega\setminus\{x_{\infty}\})$; then, in a similar way, one can also prove the $C^1_{loc}(\overline\Omega\setminus\{x_\infty\})$ convergence, getting \eqref{tesinaSTEP3}, we omit the details.
\\
Proof of \eqref{gammaeps}: let us suppose by contradiction that $\lim_{n\to+\infty}\frac{\varepsilon_n}{\gamma_{i,n}}=c<+\infty$.
	Then by \eqref{tesinaSTep1} and \eqref{F_nopiccolo}
	\begin{equation}
	\label{vipgammaip}
	\frac{v_{i,n}(x)}{\gamma_{i,n}} = \frac{E_n(x)}{\gamma_{i,n}}  +2\pi c\sum_{j=1}^2 a_j^i \frac{\partial G}{\partial y_j}(x,x_\infty)+o(1),
	\end{equation}
Observe that by \eqref{star} we have
	\begin{equation}\label{star2}
	\int_{\partial\Omega}{p_n}\frac{\partial u_{p_n}}{\partial x_j}\frac{\partial v_{i,n}}{\partial\nu}d\sigma_x=(1-\lambda_{i,n})\int_\Omega p_n^2 u_{p_n}^{{p_n}-1}\frac{\partial u_{p_n}}{\partial x_j}v_{i,n}dx.
	\end{equation}
We can evaluate the l.h.s. of \eqref{star2} combining \eqref{convergenzapup} with \eqref{vipgammaip}, indeed
	\begin{equation}\label{sx}
	\begin{aligned}
	\int_{\partial\Omega}&p_n\frac{\partial u_{p_n}}{\partial x_j}\frac{\partial v_{i,n}}{\partial\nu}d\sigma_x=8\pi\sqrt{e}\gamma_{i,n}\left[\int_{\partial\Omega}\frac{\partial G}{\partial x_j}(x,x_\infty)\frac{\partial G}{\partial \nu}(x,x_\infty)d\sigma_x+\right.
	\\
	&\qquad\qquad+\left.\int_{\partial\Omega} \frac{\partial G}{\partial x_j}(x,x_\infty) \frac{\partial}{\partial\nu}(2\pi c\sum_{k=1}^2 a_k^i\frac{\partial G}{\partial y_k})d\sigma_x+o(1)\right]
	\nonumber\\
	&=8\pi\sqrt{e}\gamma_{i,n}\!\Bigg[\int_{\partial\Omega}\nu_j(x)\left(\frac{\partial G}{\partial\nu}(x,x_\infty)\right)^2 \!d\sigma_x
		\\
	&\qquad\qquad+2\pi c\sum_{k=1}^2 a_k^i\!\int_{\partial\Omega}\frac{\partial G}{\partial x_j}(x,x_\infty)\frac{\partial}{\partial y_k}\frac{\partial G}{\partial \nu}(x,x_\infty)d\sigma_x+o(1)\Bigg]
	\nonumber\\
	&\stackrel{\eqref{Robin1}+\eqref{Robin2}}{=} 8\pi\sqrt e \gamma_{i,n}\left(\frac{\partial R}{\partial y_j}(x_\infty)+\pi c\sum_{k=1}^2 a^i_k \frac{\partial^2 R}{\partial x_k \partial x_j}(x_\infty)+o(1)\right)\nonumber\\
	&\stackrel{\nabla R(x_\infty)=0}{=}8\pi^2 \sqrt{e} c \gamma_{i,n}
	\left(\sum_{k=1}^2 a^i_k \frac{\partial^2 R}{\partial x_k \partial x_j}(x_\infty)+o(1)\right).
	\end{aligned}
	\end{equation}
	In order to estimate the r.h.s. of \eqref{star2} we first observe that
	\[
	\begin{aligned}
	\left| \int_{\Omega\setminus B_r(x_n)}\right. &\left.p_n^2 u_{p_n}^{p_n-1} \frac{\partial u_{p_n}}{\partial x_j} v_{i,n} dx \right|\leq p_n\|u_{p_n}\|_{L^\infty(\Omega\setminus B_r(x_n))}^{p_n-1}\int_{\Omega\setminus B_r(x_n)}p_n\left|\frac{\partial u_{p_n}}{\partial x_j}\right| dx\\
	&\stackrel{\eqref{utile}+\eqref{convergenzapup}}{\leq} o(1)\left(8\pi\sqrt e\|\frac{\partial G}{\partial x_j}\|_{L^\infty(\Omega\setminus B_{\frac{r}{2}}(x_\infty))}|\Omega|+o(1)\right)=o(1).
	\end{aligned}
	\]
	Hence
	\begin{equation}
	\label{dx}
	\begin{aligned}
	\int_{\Omega} p_n^2 u_{p_n}^{p_n-1}&\frac{\partial u_{p_n}}{\partial x_j} v_{i,n} dx=
	\int_{B_r(x_n)} p_n^2 u_{p_n}^{{p_n}-1}\frac{\partial u_{p_n}}{\partial x_j} v_{i,n} dx + o(1)\\
	&=\frac{1}{\varepsilon_n}\int_{B_{\frac{r}{\varepsilon_n}}(0)}\left(1+\frac{w_n}{p_n}\right)^{p_n-1}\!\!\! p_n\frac{\partial \widetilde u_{p_n}}{\partial z_j}\widetilde v_{i,n} dz+o(1)
	\\
	&\stackrel{(\star)}{=}\frac{1}{\varepsilon_n}\left(\int_{\R^2}e^{U}\sqrt{e}\frac{\partial U}{\partial z_j} \frac{a_1^i z_1+a_2^i z_2}{8+|z|^2}dz+o(1)\right)
	\\
	&=\frac{1}{\varepsilon_n}\left(-\frac{\sqrt e \pi}{3}a^i_j+o(1)\right),
	\end{aligned}
	\end{equation}
	by explicit computation, while the passage to the limit in $(\star)$ follows  by \eqref{convRiscalateNeiMax} and Remark \ref{rem:lebesgue}. Indeed defining  
	$g_n(z):=p_n \frac{\partial \widetilde u_{p_n}}{\partial z_j}\widetilde v_{i,n}$,  one can take   $g(z):=\sqrt{e}\frac{\partial U}{\partial z_j} \frac{a_1^i z_1+a_2^i z_2}{8+|z|^2}$  thanks to
	\eqref{convnablautilde} and \eqref{vtildeasenzab},
	and  $h(z):=\frac{C}{|z|}$ by \eqref{P4tilde}.
	\\	
	Putting together \eqref{star2}, \eqref{sx} and  \eqref{dx} we get
	\begin{equation}
	\label{beforefirst}
	8\pi^2 \sqrt{e} c \gamma_{i,n}\left(	\sum_{k=1}^2 a^i_k \frac{\partial^2 R}{\partial x_j \partial x_k}(x_\infty)+o(1)\right)=(1-\lambda_{i,n})\frac1{\varepsilon_n}(-\frac{\sqrt{e}\pi}{3}a_j^i+o(1))
	\end{equation}
	and finally
	\begin{equation}
	\label{first}
	(\lambda_{i,n}-1)=\frac{24\pi c \gamma_{i,n}\varepsilon_n}{a_j^i}\sum_{k=1}^2 a_k^i\frac{\partial^2 R}{\partial x_j \partial x_k}(x_\infty)\,(1+o(1)),
	\end{equation}
	for $j$ such that $a_j^i\neq 0$.
	
	Now we consider the Pohozaev identity \eqref{integralidentity} computed at the point $x_n$:
	\begin{equation}\label{integralidentityNOSTRA}
	\int_{\partial\Omega}(x-x_n)\cdot\nabla u_{p_n}\frac{\partial v_{i,n}}{\partial\nu}d\sigma_x=(1-\lambda_{i,n})p_n\int_\Omega  u_{p_n}^{{p_n}-1}v_{i,n}\left((x-x_n)\cdot\nabla u_{p_n}+\frac{2}{{p_n}-1}u_{p_n}\right)dx.
	\end{equation}
Passing into the limit in the l.h.s. of \eqref{integralidentityNOSTRA} and using \eqref{vipgammaip} and \eqref{convergenzapup}, we find 
	\begin{eqnarray}
	\label{sx2}
	&&\!\!\!\!\!\!\!\!\!\!\!\!\!\!\!\!\!\!\!\!\!\!\!\!\!\!\!\!\!\!\int_{\partial\Omega} (x-x_n)\cdot \nabla u_{p_n}\frac{\partial v_{i,n}}{\partial \nu}dx=
	\nonumber\\
	&=&
	\frac{\gamma_{i,n}}{p_n}\int_{\partial \Omega}
	8\pi\sqrt{e}(x-x_\infty)\cdot\nabla G (x,x_\infty)\frac{\partial G}{\partial \nu}(x,x_\infty) d\sigma_x
	\nonumber\\
	&& +\, \frac{\gamma_{i,n}}{p_n}\left(
	2\pi c\sum_{j=1}^2 a_j^i\int_{\partial \Omega}8\pi\sqrt{e} (x-x_\infty) \cdot\nabla G (x,x_\infty)
	\frac{\partial^2 G}{\partial \nu\partial y_j}(x,x_\infty) d\sigma_x\, +o(1)\right)
	\nonumber\\
	&\stackrel{\eqref{Green1}+\eqref{Green2}}{=}&\frac{\gamma_{i,n}}{p_n}(4\sqrt{e})+8\pi^2\sqrt{e} c\sum_{j=1}^2 a_j^i\frac{\partial R}{\partial y_j}(x_\infty)+o(1))
	\nonumber\\
	&\stackrel{\nabla R(x_\infty)=0}{=}&\frac{\gamma_{i,n}}{p_n}(4\sqrt{e}+o(1)).
	\end{eqnarray}
	Concerning the r.h.s. of \eqref{integralidentityNOSTRA}, the same computations performed in \eqref{ago4} and \eqref{ago5}, with $b^i=0$, give
	\begin{equation}
	\label{dx2}
	(1-\lambda_{i,n})\int_\Omega p_n u_{p_n}^{p_n-1}v_{i,n}\left((x-y)\cdot\nabla u_p+\frac2{p_n-1} u_{p_n}\right)dx=(1-\lambda_{i,n})o(1).
	\end{equation}
	By \eqref{sx2} and \eqref{dx2} we deduce
	\[
	\frac{\gamma_{i,n}}{p_n}(4\sqrt{e}+o(1))=(1-\lambda_{i,n})o(1),
	\]
	and in turn by \eqref{first}
	\[
	-\frac{\gamma_{i,n}}{p_n}(4\sqrt{e}+o(1))=\frac{24\pi c \gamma_{i,n}\varepsilon_n}{a_j^i}\sum_{k=1}^2 a_k^i\frac{\partial^2 R}{\partial x_j \partial x_k}(x_\infty)\,o(1)
	\]
	which is impossible. This proves that  $\gamma_{i,n}=o(\varepsilon_n)$.	
	\\\\\\
	{\emph{Step 4. Conclusion of the proof of \eqref{ele2} in Theorem \ref{teoAutofunzioniAutovalori} }}
	
	\
	
	From \eqref{tesinaSTep1}, \eqref{F_nopiccolo} and \eqref{tesinaSTEP3} we get \eqref{ele2}.
\end{proof}

\

\

\subsection{Proof of \eqref{ele1} in Theorem \ref{teoAutofunzioniAutovalori}}

\begin{proof}
	We estimate the behavior of $(\lambda_{i,n}-1)$ using \eqref{star2}.\\
	Arguing as in the proof of \eqref{ele2} in Theorem \ref{teoAutofunzioniAutovalori}, where the l.h.s. of \eqref{star2} is estimated, but using \eqref{ele2} itself instead of \eqref{vipgammaip}, we obtain 
	\[
	\lambda_{i,n}-1 =\frac{24\pi\varepsilon_n^2}{a_j^i}\sum_{k=1}^2 a_k^i\frac{\partial^2 R}{\partial x_j \partial x_k}(x_\infty)\,(1+o(1))
	\]
	so
	\[
	\frac{1-\lambda_{i,n}}{\varepsilon_n}\to 24\pi\eta_i
	\]
	where $\eta_i=\left(\sum_{k=1}^2 a_k^i\frac{\partial^2 R}{\partial x_j \partial x_k}(x_\infty)\right)/a_j^i$ for $j$ such that $a_j^i\neq0$.\\
	Moreover we have
	\begin{equation}\label{etaieigenvalue}
	\sum_{k=1}^2 a_k^i\frac{\partial^2 R}{\partial x_j \partial x_k}(x_\infty)=\eta_i a_j^i
	\end{equation}
	both if $a_j^i\neq0$ and if $a_j^i=0$ by the analogous of \eqref{beforefirst} once \eqref{vipgammaip} is substituted by \eqref{star2}.\\
	From \eqref{etaieigenvalue} we have that $\eta_i$ is an eigenvalue of $D^2R(x_\infty)$, the hessian matrix of the Robin function $R$ at the point $x_{\infty}$ with $a^i$ as corresponding eigenvector.\\
	By Proposition \ref{thm:eigenfunctions} the eigenvectors $a^i$ are orthogonal, thus the numbers $\eta_i$ are the eigenvalues $\mu_1\leq\mu_2$ of $D^2 R(x_\infty)$. In particular, since $\lambda_{2,n}\leq \lambda_{3,n}$, then $\eta_2= \mu_1$ and $\eta_3= \mu_2$.
\end{proof}

\

\

\section{Proof of \eqref{ele4} and \eqref{ele3} of Theorem \ref{teoAutofunzioniAutovalori} and proof of Theorem \ref{theorem:Morse1bubble}}\label{section:ele34} 

In this section we study the asymptotic behavior as $n\rightarrow +\infty$ of the fourth eigenvalue $\lambda_{4,n}$ and fourth eigenfunction   $v_{4,n}$ of the linearized problem \eqref{eigenvalueProblem}, proving \eqref{ele4} and \eqref{ele3} of Theorem \ref{teoAutofunzioniAutovalori}. A the end of the section, using the results in Theorem \ref{teoAutofunzioniAutovalori}, we then prove Theorem \ref{theorem:Morse1bubble}.

\subsection{Proof of \eqref{ele4} and \eqref{ele3} of Theorem \ref{teoAutofunzioniAutovalori}}

\begin{proof}	By the variational characterization of the eigenvalues
	\begin{equation}\label{lun0}
	\lambda_{4,n}=\inf_{\substack{
			v\in H^1_0(\Omega),\,v\not\equiv0\\ v\perp\{v_{1,n},v_{2,n},v_{3,n}\}}}\frac{\int_\Omega |\nabla v|^2 dx}{p_n\int_\Omega u_{p_n}^{p_n-1}v^2 dx}.
	\end{equation}
	Let 
	\[\psi_{4,n}(x):=(x-x_n)\cdot\nabla u_{p_n}(x)+\frac{2}{p_n-1}u_{p_n}(x)\] and let us define the function
	\[v:=\widehat\phi_{n}\psi_{4,n}+a_{1,n}v_{1,n}+a_{2,n}v_{2,n}+a_{3,n}v_{3,n},\] with	
\[
\widehat\phi_{n}(x):=\left\{
\begin{array}{ll}
1&\mbox{if $|x-x_n|\leq\varepsilon_n$}\\
\frac{1}{\log\frac{\varepsilon_n}{r}}\log\frac{|x-x_n|}{r}&\mbox{if $\varepsilon_n<|x-x_n|\leq r$}\\
0&\mbox{if $|x-x_n|>r$}
\end{array}
\right.
\]
(for $r>0$ as in the statement of Theorem \ref{teoAsymptotic}) and
	\[
	a_{i,n}:=-\frac{\int_\Omega p_n  u_{p_n}^{p_n-1}\widehat\phi_{n}\psi_{4,n}v_{i,n}dx}{p_n\int_\Omega u_{p_n}^{p_n-1}v_{i,n}^2 dx}=-\frac{N_{i,n}}{D_{i,n}}\qquad i=1,2,3.
	\]
Observe that by definition $v\perp \{v_{1,n},v_{2,n},v_{3,n}\}$ in $H^1_0(\Omega)$.\\\\
{\emph{Step 1. For any $i=1,2,3$ we show that $a_{i,n}=o(1)$, by proving that $N_{i,n}=o(1)$ and that there exist $d_i>0$ such that $D_{i,n}\geq d_i>0$.}}	
\\\\
Recalling that $\|v_{i,n}\|_\infty=1$ and observing that $\|\psi_{4,n}\|_{L^{\infty}(\Omega)}=o(1)$ by \eqref{P4} and \eqref{boundLoo}, we can easily estimate $N_{i,n}$:
\begin{equation*}
|N_{i,n}|=p_n\left|\int_\Omega u_{p_n}^{p_n-1}v_{i,n}\widehat\phi_n \psi_{4,n}\,dx\right|\leq o(1) \,p_n\int_\Omega u_{p_n}^{p_n-1}\,dx\stackrel{\eqref{bound_int_p-1}}{=}o(1).
\end{equation*}
\\\\
Next we estimate $D_{i,n}$, for $i=1,2,3$.\\
Being $v_{1,n}=u_{p_n}$, which by assumption is a $1$-spike-sequence of solutions:
\[
D_{1,n}=p_n\int_\Omega u_{p_n}^{p_n+1}dx\overset{\eqref{problem}}{=}p_n\int_\Omega |\nabla u_{p_n}|^2dx\overset{\eqref{convergenzaenergia}}{\longrightarrow}8\pi e=:2d_1.
\]
While for $i=2,3$, by \eqref{utile} and $\|v_{i,n}\|_{\infty}=1$, (taking $r>0$ as in the statement of Theorem \ref{teoAsymptotic}) we get
	\begin{eqnarray*}
D_{i,n}&=&p_n\int_{B_r(x_n)} u_{p_n}^{p_n-1}v^2_{i,n}\,dx+o(1)\\
&=&\int_{B_{\frac{r}{\varepsilon_n}}(0)}\left(1+\frac{w_{n}(y)}{p_n}\right)^{p_n-1}\tilde{v}_{i,n}^2(y)\, dy+o(1)\\
&\overset{(\star)}{\longrightarrow}&\int_{\R^2} e^{U(y)}\left(\sum_{j=1}^2\frac{a^i_j y_j}{8+|y|^2}\right)^2
dy=:2 d_i>0,
\end{eqnarray*}
where in order to pass to the limit into $(\star)$ we have used Proposition \ref{thm:eigenfunctions} and Remark \ref{rem:lebesgue} (with $g_n=\tilde v_{i,n}^2$, $|g_n|\leq1$).

\

{\emph{Step 2.}}	Using that $v_{i,n}$ satisfies the eigenvalue problem  \eqref{eigenvalueProblem} and that $v_{i,n}$ and $v_{j,n}$ are orthogonal in $H^1_0(\Omega)$ if $i\neq j$, we have
	\begin{eqnarray*}
		\int_\Omega |\nabla v|^2dx&=&\int_\Omega |\nabla (\widehat\phi_{n}\psi_{4,n})|^2dx
		+2\sum_{i=1}^3 \lambda_{i,n}a_{i,n}N_{i,n}+\sum_{i=1}^3\lambda_{i,n}a_{i,n}^2D_{i,n}
		\\
		&\overset{Step\, 1}{=}&\int_\Omega |\nabla (\widehat\phi_{n}\psi_{4,n})|^2dx+o(1).
	\end{eqnarray*}
	Finally, since $\psi_{4,n}$ solves the linearized equation \eqref{eigenvalueProblem}, we have
	\begin{equation}
	\label{lun1}
	\int_\Omega |\nabla v|^2 dx=\int_\Omega \psi_{4,n}^2|\nabla \widehat\phi_{n}|^2dx+p_n\int_\Omega u_{p_n}^{p_n-1} \psi_{4,n}^2 \widehat\phi_{n}^2 dx+o(1).
	\end{equation}
	In a similar way we get
	\begin{equation}
	\label{lun2}
	p_n\int_\Omega u_{p_n}^{p_n-1} v^2 dx=p_n\int_\Omega u_{p_n}^{p_n-1} \psi_{4,n}^2 \widehat\phi_{n}^2 dx+o(1),
	\end{equation}
	therefore inserting \eqref{lun1} and \eqref{lun2} into \eqref{lun0} we get
	\[
	\lambda_{4,n}\leq1+\frac{\int_\Omega \psi_{4,n}^2|\nabla \widehat\phi_{n}|^2dx+p_n\int_\Omega u_{p_n}^{p_n-1} \psi_{4,n}^2 \widehat\phi_{n}^2 dx+o(1)}{p_n\int_\Omega u_{p_n}^{p_n-1} \psi_{4,n}^2 \widehat\phi_{n}^2 dx+o(1)}.
	\]
	Let us estimate the last two integrals
	\begin{eqnarray*}
		\int_\Omega \psi_{4,n}^2|\nabla \widehat\phi_{n}|^2dx&=&\frac{1}{\log^2(\frac{\varepsilon_n}{r})}\int_{\{\varepsilon_n\leq|x-x_n|\leq r\}}\left((x-x_n)\cdot\nabla u_{p_n}+\frac{2}{p_n-1}u_{p_n}\right)^2\frac{dx}{|x-x_n|^2}\\
		&=&\frac{1}{\log^2(\frac{\varepsilon_n}{r})}\int_{\{1\leq|y|\leq \frac r{\varepsilon_n}\}}\left(y\cdot\nabla\widetilde u_{p_n}+\frac{2}{p_n-1}\widetilde u_{p_n}\right)^2\frac{dy}{|y|^2}\\
		&\overset{\eqref{P4tilde}+\eqref{boundLootilde}}{\leq}&\frac{C}{p^2\,\log^2(\frac{\varepsilon_n}{r})}\int_1^{\frac{\varepsilon_n}{r}}\frac{ds}{s}\\
		&=&\frac{C}{p^2\,\log(\frac{\varepsilon_n}{r})}=\frac{o(1)}{p^2}.
	\end{eqnarray*}
	\begin{eqnarray*}
		p_n\int_\Omega u_{p_n}^{p_n-1} \psi_{4,n}^2 \widehat\phi_{n}^2 dx&=&
			p_n\int_{B_r(x_n)} u_{p_n}^{p_n-1} \psi_{4,n}^2 \widehat\phi_{n}^2 dx
\\
&=&\frac{1}{p_n^2}\int_{B_{\frac{r}{\varepsilon_n}}(0)}\!\!\!\left(1+\frac{w_n}{p_n}\right)^{p_n-1}\!\!\!\!\!\!\!\!\!p_n^2\left(y\cdot \nabla \widetilde u_{p_n}+\frac{2}{p_n-1}\widetilde u_{p_n}\right)^2\widehat\phi^2_n(x_n+\varepsilon_ny)dy\\
	&\overset{(\star)}{=}&\frac{1}{p_n^2}\left(\int_{\R^2}e^{U(y)}e(y\cdot \nabla U(y)+2)^2dy+o(1)\right)\\
		&\overset{\eqref{definizioneU}}{=}&\frac{1}{p_n^2}\left(\int_{\R^2}e^{U(y)}4e\left(\frac{8-|y|^2}{8+|y|^2}\right)^{2}dy+o(1)\right)\\
		&=&\frac{1}{p_n^2}\left(\frac{32}{3}\pi e+o(1)\right),
	\end{eqnarray*}

where in order to pass to the limit into $(\star)$ we have used Remark \ref{rem:lebesgue}, setting  $g_n(y):=p_n^2\left( y\cdot\nabla \widetilde u_{p_n}(y)+\frac{2}{{p_n}-1}\widetilde u_{p_n}(y)\right)^2\widehat\phi^2_n(x_n+\varepsilon_ny)$. Indeed, 
by \eqref{convutilde} and \eqref{convnablautilde}   one  has $g(y):= e(y\cdot \nabla U(y)+2)^2$   
and, by \eqref{P4tilde} and \eqref{boundLootilde}, one can take $h$ to be a constant function.

\

We have shown so far that $\lambda_{4,n}\leq 1+o(1)$ and hence $\lambda_{4,n}\to1$ as $n\to+\infty$.
	By Lemma \ref{lemma:limab}:
	\begin{equation*}
	\widetilde{v}_{4,n}(y)\longrightarrow\sum_{j=1}^2\frac{a^4_j y_j}{8+|y|^2}+ b^4\frac{8-|y|^2}{8+|y|^2}\qquad \mbox{as $n\to+\infty$\quad in $C^1_{loc}(\R^2)$}
	\end{equation*}
	with $(a^4_1,a^4_2,b^4)\neq(0,0,0)$. Let, for $i=2,3$, $a^i=(a_1^i,a_2^i)$ be as in \eqref{vtildeasenzab}. The eigenfunctions $v_{4,n}$ and $v_{i,n}$ are orthogonal in $H^1_0(\Omega)$ for $i=2,3$, then by using Proposition \ref{thm:eigenfunctions} we have that the vector $(a_1^4,a_2^4)$ is orthogonal to $(a_1^2,a_2^2)$ and $(a_1^3,a_2^3)$, which are both different from zero and orthogonal. This implies $(a_1^4,a_2^4)=(0,0)$ and so $b^4\neq0$. We are then in position to apply Proposition \ref{prop:assurdob} getting \eqref{ele3} and \eqref{ele4}, where $b:=-b^4$.
\end{proof}

\

\subsection{Proof of Theorem \ref{theorem:Morse1bubble}}
\begin{proof} 
	By \eqref{ele1} in Theorem \ref{teoAutofunzioniAutovalori} it follows that, for $i=2,3$: if $\mu_{i-1}<0$ then $\lambda_{i,n}<1$ for $n$ large, and conversely if $\lambda_{i,n}< 1$ for $n$ large then $\mu_{i-1}\leq 0$. As a consequence, recalling also that $\lambda_{i,n}=\frac{1}{p_n}<1$, we get:
	\[
	1+m(x_\infty)\leq m(u_{p_n})\leq m_0(u_{p_n})\leq 1+m_0(x_\infty), \quad\mbox{for $n$ large},
	\]
	where $m(u_{p_n})\leq m_0(u_{p_n})$ by definition. The thesis follows observing that
	$\mu_1+\mu_2=\Delta R(x_\infty)$ and using the strict subharmonicity of the Robin function  (see \eqref{Robinequation} in Lemma \ref{thm:Robin}), which implies that $m_0(x_\infty)\leq 1$.
\end{proof}

\

\

\section{The case $\Omega$ convex}\label{section:conclusion}

Along this section we assume that $\Omega$ is convex and prove Corollary \ref{teoMorse}.

\

\begin{proof}[Proof of Corollary \ref{teoMorse}] $\;$\\
	We argue by contradiction assuming that there exist $C>0$ and a sequence $u_{p_n}$, $p_n\to+\infty$, of solutions to \eqref{problem} such that $\sup_n p_n \|\nabla u_{p_n}\|^2_{2}\leq C$ and 
	\begin{equation}
	\label{last}
	m(u_{p_n})\neq 1\quad \mbox{or}\quad u_{p_n} \mbox{  is degenerate}\quad \forall\:n.
	\end{equation}
Applying Theorem \ref{teoAsymptotic}, we have that there exist $k\in\mathbb N$, a set $\mathcal S\subset\Omega$ of $k$ points and a subsequence $u_{p_{n_j}}$ satisfying all conditions \eqref{pu_va_a_zeroTeo}---\eqref{P4}. \\
	Since $\Omega$ is convex, we can apply Theorem 2.4 in  \cite{GrossiTakahashi} to deduce that  $u_{p_{n_j}}$ is a $1$-spike-sequence of solutions, i.e. $k=1$, so that the concentration set $S$
	reduces to a single point that we denote by $x_{\infty}$. 
	By \eqref{x_j relazioneTeo}, $x_{\infty}\in \Omega$ is a critical point of the  Robin function $R$.
	\\
	Since $\Omega$ is a convex bounded domain,  by Lemma \ref{thm:Robin2} it follows that $x_{\infty}$ is nondegenerate and that $m(x_\infty)=0$. \\
	So, by Theorem \ref{theorem:Morse1bubble}, there exists $j^{*}\in\mathbb N$ such that  
	\begin{equation*}
m(u_{p_{n_j}})=1\quad\mbox{and}\quad u_{p_{n_j}}\mbox{ is nondegenerate}, \quad \forall j\geq j^{*}.\end{equation*}
	This clearly contradicts \eqref{last}.
\end{proof}

\

We conclude the section pointing out once more that, since in convex domains \eqref{energylimit} is equivalent to the uniform bound \eqref{boundSirakov} proved recently in \cite{Sirakov}, then the Morse index of any solution of \eqref{problem} in any convex domain is one if the exponent $p$ is large. This gives Theorem \ref{teoUnicita} as described in the Introduction.
\

\

	\end{document}